\documentclass{article}
\usepackage[utf8]{inputenc}

\usepackage{float}
\usepackage{graphicx}
\usepackage{amsmath}
\usepackage{enumitem}
\usepackage{amscd}
\usepackage{amsthm}
\usepackage{mathtools}
\usepackage{amssymb}
\usepackage{tikz-cd}
\usepackage{caption}
\usepackage{dsfont}
\usepackage[font=small,labelfont=bf,width=1\textwidth]{caption}
\usepackage{multicol}
\usepackage{subfig}
\usepackage[left=3.5cm,top=4cm,right=3.5cm,bottom=4cm]{geometry}

\usepackage{siunitx}
\usepackage{lipsum}
\usepackage{xfrac}
\usepackage{marvosym}
\usepackage{epic}
\usepackage{wasysym}
\usepackage{epstopdf}
 
\usepackage{hyperref}
\usepackage{wrapfig}
\usepackage{enumitem}
\usepackage{enumerate}
\usepackage{color}

\theoremstyle{plain}
\newtheorem{theorem}{Theorem}[section]
\newtheorem{lemma}[theorem]{Lemma}
\newtheorem{corollary}[theorem]{Corollary}
\newtheorem{proposition}[theorem]{Proposition}

\newtheorem{remark}[theorem]{Remark}

\theoremstyle{definition}
\newtheorem{definition}[theorem]{Definition}

\mathtoolsset{showonlyrefs}

\newcommand{\bbH}{\mathbb H}
\newcommand{\Tr}{\mathrm{Tr}}
\newcommand{\Rea}{\mathrm{Re}}

\title{Quaternionic stochastic areas on quaternionic full flag manifolds and applications}

\author{Fabrice Baudoin\footnote{Research partially supported by grant 10.46540/4283-00175B from Independent Research Fund Denmark and by the Villum Investigator grant \emph{Stochastic Analysis in Aarhus}. F.B. also acknowledges funding from the European Research Council (ERC) under the European Union’s Horizon Europe research and innovation programme (RanGe project, Grant Agreement No. 101199772).}, Teije Kuijper\footnote{Research was completed while visiting Purdue University.}, Jing Wang\footnote{Research was supported in part by NSF Grant DMS-2246817.}}
\begin{document}

\maketitle

\begin{abstract}
We show that a Brownian motion on the quaternionic full flag manifold can be represented as a matrix-valued diffusion obtained in a simple way from a symplectic Brownian motion.
By relating its radial dynamics to the Brownian motion on the quaternionic sphere, an explicit formula for the characteristic function of the joint distribution of the quaternionic stochastic areas is obtained.
The limit law of these quaternionic stochastic areas is shown to be a multivariate normal distribution with zero mean and non-diagonal covariance matrix.
These results are subsequently applied to establish new results about simultaneous quaternionic windings on the quaternionic spheres.
\end{abstract}

\tableofcontents

\newpage

\section{Introduction}

Stochastic area processes associated with Brownian motion on manifolds have played an
important role in probability on Lie groups, geometric analysis, and the study of integrable 
Brownian functionals arising from differential geometry.  Their analysis
began with Lévy's \cite{Levy1951} celebrated description of stochastic areas on the plane, and was later extended to the complex Hopf fibration by two of the authors in \cite{MR3719061}, and to the quaternionic Hopf fibration in the subsequent work \cite{baudoin2014}.  We refer to the recent monograph \cite{Baudoin2024-is} for an overview of the state of the art.

In the complex case, the contact geometry of the Hopf fibration and of complex flag manifolds provides a natural framework for defining and studying such functionals.  In particular, the stochastic area processes on the full complex flag manifold have recently  been introduced and studied in \cite{baudoin2025fullflag}.  The quaternionic
case, although analogous at a formal level, presents a fundamentally different set of
challenges.  The absence of a non-trivial canonical quaternionic calculus (see \cite{Buff1973}) means that
quaternionic Itō-type computations quickly become intractable.  While non-standard
quaternionic calculi do exist, see for example \cite{Dzagnidze2012}, they do not lend themselves naturally to the analysis of Brownian motions on quaternionic manifolds.  As a consequence, explicit stochastic area formulas beyond the quaternionic Hopf fibration have remained largely unexplored.

One of the aims of this paper is to initiate the theory of quaternionic stochastic areas to the setting of quaternionic flag manifolds, with emphasis on the \emph{quaternionic full flag manifold} which parametrizes nested sequences of quaternionic subspaces
\[
\{0\} \subsetneq W_1 \subsetneq \cdots \subsetneq W_{n-1} \subsetneq \mathbb{H}^n.
\]
Quaternionic flag manifolds arise naturally in algebraic topology \cite{MARE20082830, Mare_Willems_2009}, multisymplectic geometry \cite{FOTH2002330, Foth2003} and even  image processing \cite{wang2025colorimagesetrecognition}.
The full flag manifold plays a particularly important role, since all quaternionic partial flag manifolds may be obtained from it by canonical Riemannian submersions.  

The identification
\[
F_{1,2,\dots,n-1}(\mathbb{H}^n)
  \simeq \mathrm{Sp}(n)/\mathrm{Sp}(1)^n
\]
as a Riemannian homogeneous space provides a representation of the full flag well suited to stochastic analysis computations. Our first contribution is the construction of Brownian motion on the quaternionic full
flag manifold using the symplectic Brownian motion on $\mathrm{Sp}(n)$ and its projection under the Riemannian submersion
\[
\mathrm{Sp}(1)^n \longrightarrow \mathrm{Sp}(n) \longrightarrow F_{1,2,\dots,n-1}(\mathbb{H}^n).
\]
Our approach, and this is one of the main technical novelties compared to \cite{baudoin2014} of the paper, avoids quaternionic Itō's formula computations entirely.  Instead, we only work with the geometry of the above fibration.  For instance, a key
feature of this construction is that the projection of a symplectic Brownian motion onto its last row yields a Brownian motion on the quaternionic sphere $\mathbb{S}^{4n-1}$.  This
representation allows us to identify the \emph{radial} dynamics of the flag Brownian motion
explicitly and we find that the squared radial variables form a Jacobi diffusion on the simplex of index
$(3/2,\dots,3/2)$.  

We then introduce the \emph{quaternionic stochastic area processes} associated with Brownian
motion on the flag manifold.  These are defined using the connection form of the above
fibration and generalize the quaternionic area process on the Hopf fibration.  Using a
local trivialization compatible with the quaternionic affine coordinates on the flag
manifold, we show that the horizontal lift of Brownian motion admits a skew-product
representation whose vertical component is described precisely by these area processes.
This also yields an interesting and new skew-product decomposition for the symplectic Brownian motion itself.

With this in place, we derive an explicit formula for the joint characteristic
function of the quaternionic stochastic areas.  The computation combines the above
skew-product representation, the spectral decomposition of Jacobi semigroups on simplices,
and a Yor-type exponential transform adapted to the quaternionic setting.  The resulting
expression reveals non-trivial interactions between the components of the area vector,
reflecting the inherent non-commutativity of quaternionic multiplication.

Our second main result is the description of the large-time behaviour of the quaternionic
stochastic areas.  After an appropriate renormalization, the vector of area processes
converges in distribution to a centered multivariate normal law with explicitly computable,
non-diagonal covariance matrix.  Unlike in the complex case, where the limiting components
are independent, the quaternionic limit exhibits a full correlation structure governed by
the geometry of the fibration.  This is one of the striking manifestations of the difference between
complex and quaternionic stochastic areas.

Finally, we apply these results to the study of \emph{simultaneous quaternionic winding} of
Brownian motion on the sphere $\mathbb{S}^{4n-1}\subset\mathbb{H}^n$.  By expressing winding
functionals in terms of the quaternionic stochastic areas, we obtain a multivariate central limit theorem for the winding on the quaternionic sphere.
This extends and generalises the classical results for real and complex windings and provides the first such result in the
quaternionic setting beyond the Hopf fibration.

\medskip

\noindent

The paper is organized as follows: section \ref{sec:preliminairies} briefly reviews quaternions, the compact symplectic group and the quaternionic full flag manifolds, it also constructs a Brownian motion on it as a projection of a symplectic Brownian motion and investigates its radial dynamics; section \ref{sec:quaternionic-area-processes} introduces the quaternionic stochastic area processes and the quaternionic winding processes and derives a relation between them, furthermore it gives an explicit expression for the characteristic function of the winding process and determines its asymptotics; and finally section \ref{sec:quaternionic-winding-sphere} applies these results to simultaneous quaternionic windings on the quaternionic spheres, in particular establishing its asymptotics.

\section{Preliminaries and Brownian motion on the quaternionic full flag}\label{sec:preliminairies}

\subsection{Quaternions}

We start by recalling the basics of the quaternionic field and setting up the notation.
The \emph{quaternions} are  defined as the non-commutative algebra
\begin{align*}\index{$\mathbb{H}$ (quaternions)}
\mathbb{H} =\{ t+x\mathbf i+y\mathbf j+z\mathbf k\mid t,x,y,z\in\mathbb{R}\} ,
\end{align*}where $\mathbf i,\mathbf j,\mathbf k$ satisfy the relations $\mathbf i^2=\mathbf j^2=\mathbf k^2=-1$, $\mathbf i \mathbf j=\mathbf k=-\mathbf  j \mathbf i$ and $\mathbf i \mathbf k=-\mathbf  j=-\mathbf k \mathbf i$.\index{quaternions} 
For a quaternion $q=t+x\mathbf i+y\mathbf j+z\mathbf k$, its real part is 
\begin{align*}
\mathrm{Re}(q):=t ,
\end{align*}
while its imaginary part is the $3$-dimensional component
\begin{align*}
\mathrm{Im}(q) :=x\mathbf i+y\mathbf j+z\mathbf k.
\end{align*}
Its quaternionic conjugate is defined by
\begin{align*}
	\overline{q} :=t-x\mathbf i-y\mathbf j-z\mathbf k,
\end{align*}
and its norm by
\begin{align*}
    |q|^2 :=q\overline{q} =\overline{q} q = t^2+x^2+y^2+z^2 .
\end{align*}
In general we have 
\[
\overline{pq} =\overline{q}\ \overline{p} \quad \mbox{and}\quad \mathrm{Re}(q)=\frac{1}{2}(q+\overline{q}).
\]
As a real vector space, $\mathbb{H}$ can be identified with $\mathbb{R}^4$ in the obvious way. For quaternions $q_1=t_1+x_1\mathbf i+y_1\mathbf j+z_1\mathbf k$ and $q_2=t_2+x_2\mathbf i+y_2\mathbf j+z_2\mathbf k$, we use the  inner product 
\[
q_1 \cdot q_2:= t_1t_2+x_1x_2+y_1y_2+z_1z_2=\frac{1}{2}(q_1 \overline{q_2}+q_2 \overline{q_1}).
\]
The set of imaginary quaternions 
\[
\mathfrak{sp}(1):= \left\{ x\mathbf i+y\mathbf j+z\mathbf k \mid x,y,z \in \mathbb{R}  \right\}
\]
forms a Lie algebra with the Lie bracket $[q_1,q_2]=q_1q_2-q_2q_1$, and is isomorphic to the Lie algebra $\mathfrak{su}(2)$. Lastly, let us also recall that the unit quaternions form the Lie group
\[
\mathrm{Sp}(1):= \left\{  q \in \mathbb{H} \mid |q|=1 \right\} ,
\]
which is isomorphic to the Lie group $\mathrm{SU}(2)$.

\subsection{Compact symplectic group}

In the rest of this paper we focus on the case $n \ge 2$. Consider the compact symplectic group
\[
  \mathrm{Sp}(n)=\mathbf{U}(n,\bbH)
  :=\{\,U\in \bbH^{n\times n}\mid U^\ast U=I_n\,\},
\]
where the quaternionic adjoint is defined by $U^\ast=\overline U^{\,\mathsf T}$, and as before,
$\overline{a+b\mathbf i+c\mathbf j+d\mathbf k}=a-b\mathbf i-c\mathbf j-d\mathbf k$ denotes the quaternionic conjugate.
The Lie algebra of $\mathrm{Sp}(n)$ is
\[
  \mathfrak{sp}(n)=\{\,X\in \bbH^{n\times n}\mid X^\ast=-X\,\}.
\]
We equip $\mathrm{Sp}(n)$ with the bi-invariant Riemannian metric induced by the
Hilbert--Schmidt inner product on $\mathfrak{sp}(n)$, given by
\begin{align}\label{eq:Hilbert-Schmidt-inner-product}
  \langle X,Y\rangle_{\mathfrak{sp}(n)} \;=\; \frac{1}{2} \Rea (\Tr(X^\ast Y)),\qquad X,Y\in\mathfrak{sp}(n).
\end{align}

Let $E_{jk}$ denote the $n\times n$ quaternionic matrix with $1$ at the $(j,k)$ entry and $0$ elsewhere.
Write
\[
\mathbf e_0=1, \quad \mathbf e_1=\mathbf i,\quad \mathbf e_2=\mathbf j,
\quad \mathbf e_3=\mathbf k,
\]
and observe that 
$\overline{\mathbf e_0}=\mathbf e_0$, $\overline{\mathbf e_\alpha}=-\mathbf e_\alpha$ for $\alpha=1,2,3$.
 A real basis of $\mathfrak{sp}(n)$ is given by:

\paragraph{Off-diagonal generators ($1\leq j<k\leq n$, $a=0,1,2,3$).}
\[
  X^{(a)}_{jk}
  := E_{jk} \, \mathbf e_a - E_{kj} \, \overline{\mathbf e_a}.
\]

\paragraph{Diagonal generators ($j=1,\dots,n$, $a=1,2,3$).}
\[
  H^{(a)}_{j} \;:=\; \sqrt{2} E_{jj} \mathbf e_a.
\]
It follows that the real dimension of $\mathfrak{sp}(n)$, and hence $\mathrm{Sp}(n)$, is  $n(2n+1)$.

\

For computations it is often convenient to regard $\mathrm{Sp}(n)$ as a submanifold of the set of quaternionic matrices parametrized by real coordinates. More precisely, elements of a $n \times n$ matrix $U=(q_{ij})_{1 \le i,j \le n}$ in $\mathbb H^{n \times n}$ can be denoted by
\[
q_{ij}=t_{ij}+x_{ij}\mathbf i+y_{ij}\mathbf j+z_{ij}\mathbf k,
\]
and for any $h=a+b\mathbf i+c\mathbf j+d \mathbf k$, one can define a real vector field $h \cdot \partial_{q_{ij}}$ on  $\mathbb H^{n \times n}\simeq \mathbb{R}^{4n^2} $ by
\[
h \cdot \partial_{q_{ij}} :=a \partial_{t_{ij}}+b \partial_{x_{ij}}+c \partial_{y_{ij}}+d \partial_{z_{ij}}. 
\]
Therefore $\mathrm{Sp}(n)$ can  be seen as a real immersed Riemannian submanifold of $\mathbb R^{4n^2}$ equipped with $1/2$ the standard Euclidean inner product structure.
With this notation, one can compute the left-invariant vector fields on $\mathrm{Sp}(n)$ corresponding to the basis elements in $\mathfrak{sp}(n)$.  Recall that if $X \in \mathfrak{sp}(n)$, the corresponding left-invariant field is given by
\[
 (X f)(U)=\frac{d}{dt}\bigg|_{t=0} f(Ue^{tX}), \,\quad U\in \mathrm{Sp}(n).
 \]
After straightforward computations, we obtain:

\begin{lemma}\label{generators sp(n)}
Let $X_{jk}^{(a)}$ be the off-diagonal generators for $1\leq j<k\leq n$, $a=0,1,2,3$ and $H_j^{(a)}$ the diagonal generators for $j=1,\dots,n$, $a=1,2,3$, as before.
Then their corresponding left-invariant vector field are given by
\[
 X^{(a)}_{jk}=\sum_{\ell =1}^n ((q_{\ell j} \mathbf{e}_a) \cdot \partial_{q_{\ell k}}-(q_{\ell k} \overline{\mathbf{e}}_a) \cdot \partial_{q_{\ell j}})
 \]
and
 \[
  H^{(a)}_{j} =\sqrt{2} \, \sum_{\ell =1}^n (q_{\ell j} \mathbf{e}_a) \cdot \partial_{q_{\ell j}}
  \]
respectively.
\end{lemma}

Since  $X^{(a)}_{jk},H^{(a)}_{j}$ form a left-invariant orthonormal frame and  $\mathrm{Sp}(n)$ is a unimodular Lie  group, it follows that the Laplace-Beltrami operator on $\mathrm{Sp}(n)$ is given by

\begin{equation}\label{eq:Delta-left-sum}
  \Delta_{\mathrm{Sp}(n)}  \;=\;\sum_{1\le j<k\le n}\ \sum_{a=0}^3 \big(X_{jk}^{(a)}\big)^{\!2}
               \;+\;\sum_{j=1}^n\ \sum_{a=1}^3 \big(H_j^{(a)}\big)^{\!2}.
\end{equation}

\subsection{Quaternionic full flag manifold}\label{section full flag}

The \emph{quaternionic full flag manifold} is defined here as the Riemannian homogeneous space
\[
F_{1,2\dots ,n-1}(\mathbb{H}^n):=\mathrm{Sp}(n)/\mathrm{Sp}(1)^n,
\]
where we see $\mathrm{Sp}(1)^n$ as the subgroup of all diagonal unitary quaternionic matrices and the corresponding action is given by right multiplication.

More precisely, $\mathrm{Sp}(1)^n$ is the closed subgroup of $\mathrm{Sp}(n)$ with Lie algebra
\[
\mathfrak{sp}(1)^n=\mathrm{span} \left\{ H_j^{(a)}\bigm| a=1,2,3, \  j=1,\dots,n \right\}
\]
and the quotient space is the coset space
\[
\mathrm{Sp}(n)/\mathrm{Sp}(1)^n=\left\{ U \mathrm{Sp}(1)^n\mid U \in \mathrm{Sp}(n)  \right\}.
\]
From   \cite[Theorem 2.4.1]{MR2500106}, there exists a unique Riemannian metric on $\mathrm{Sp}(n)/\mathrm{Sp}(1)^n$ such that the projection map 
\[
\pi: \mathrm{Sp}(n) \longrightarrow \mathrm{Sp}(n)/\mathrm{Sp}(1)^n
\]is a Riemannian submersion with totally geodesic fibers.  The Laplace-Beltrami operator associated with this metric will be denoted by $\Delta_{F_{1,2\dots ,n-1}(\mathbb{H}^n)}$.

The vertical bundle $\mathcal{V}$ of this submersion  is  the sub-bundle of $T\mathrm{Sp}(n)$  tangent to the fibers. It is given by 
\[
\mathcal{V}=\{\mathcal{V}_U\mid U\in \mathrm{Sp}(n) \},\quad  
\mathcal{V}_U=\mathrm{span} \left\{ H_j^{(a)}(U) \mid a=1,2,3, \  j=1,\dots,n \right\}.
\]
The horizontal bundle, which is the orthogonal bundle to $\mathcal{V}$, is given by
\begin{align}\label{horizontal bundle}
\mathcal{H}=\{\mathcal{H}_U\mid U\in \mathrm{Sp}(n) \},\quad \mathcal{H}_U=\mathrm{span} \left\{ X_{jk}^{(a)}(U) \mid a=0,1,2,3, \ 1\leq j<k\leq n\right\}.
\end{align}
Note that  $\mathrm{Sp}(n)$ acts by isometries on $\mathrm{Sp}(n)/\mathrm{Sp}(1)^n$, so $\mathrm{Sp}(n)/\mathrm{Sp}(1)^n$ is indeed a Riemannian homogeneous space. Moreover, we mention that the resulting fibration
\[
\mathrm{Sp}(1)^n \longrightarrow \mathrm{Sp}(n) \longrightarrow \mathrm{Sp}(n)/\mathrm{Sp}(1)^n
\]
is a special case of a B\'erard-Bergery fibration, see \cite[Theorem 9.80]{Besse}.

The name quaternionic full flag manifold comes from the fact that it can also be defined as the space of full quaternionic flags in the right skew-vector space $\mathbb{H}^n$.
By a \emph{full quaternionic flag} we mean a sequence of quaternionic vector subspaces
\begin{align*}
    \{ 0\}\subsetneq W_1\subsetneq\dots\subsetneq W_{n-1}\subsetneq\mathbb{H}^n .
\end{align*}
It can be easily verified that the full quaternionic flag is diffeomorphic to the quotient space $\mathrm{Sp}(n)/\mathrm{Sp}(1)^n$. Consider the surjective map
\begin{align*}
    \mathrm{Sp}(n)\longrightarrow &\hspace{1.5mm} F_{1,2\dots ,n-1}(\mathbb{H}^n)\\
    M\quad\longmapsto & (W_1,\dots ,W_{n-1}) ,
\end{align*}
where $W_j=\mathrm{span}_{\mathbb{H}}(Me_1,\dots ,Me_j)$ and $e_1,\dots ,e_n$ is the canonical basis of $\mathbb{H}^n$.
This map is invariant under the action of $\mathrm{Sp}(1)^n$ on $\mathrm{Sp}(n)$ and thus descends to a bijection $\mathrm{Sp}(n)/\mathrm{Sp}(1)^n\rightarrow F_{1,2,\dots ,n-1}(\mathbb{H}^n)$. 

We now parametrize (a dense subset of) $F_{1,2,\dots ,n-1}(\mathbb{H}^n)$ using  \emph{quaternionic local affine} variables as follows. Let 
\begin{align}\label{eq:quaternionic-domain}
\mathcal{D} :=\left\{ \left(\begin{matrix}
        q_{11} & \dots & q_{1n} \\
        \vdots & \ddots &\vdots \\
        q_{n1} & \dots & q_{nn}
    \end{matrix} \right) \in \mathrm{Sp}(n) \biggm| q_{n1} \neq 0, \dots, q_{nn} \neq 0   \right\}
\end{align}
and consider the smooth map $p: \mathcal{D} \to \mathbb{H}^{n-1} \times \dots \times  \mathbb{H}^{n-1}$ defined by
\begin{align}\label{submersion p}
p\left(\begin{matrix}
        q_{11} & \dots & q_{1n} \\
        \vdots & \ddots &\vdots \\
        q_{n1} & \dots & q_{nn}
    \end{matrix} \right) 
    =\left( \left(\begin{matrix}
        q_{11}q^{-1}_{n1} \\
        \vdots \\
        q_{(n-1)1}q^{-1}_{n1}
        \end{matrix}
        \right), \dots, \left(\begin{matrix}
        q_{1n}q^{-1}_{nn} \\
        \vdots \\
        q_{(n-1)n}q^{-1}_{nn}
        \end{matrix}
        \right) \right).
\end{align}
Clearly for every $M_1,M_2 \in \mathcal{D}$, $p(M_1)=p(M_2)$ is equivalent to $M_2=M_1g$ for some $g \in \mathrm{Sp}(1)^n$. Since $p$ is a submersion from $\mathcal{D}$ onto its image $\mathcal{O}:=p(\mathcal{D})$, one deduces that there exists a diffeomorphism $\Phi$ between an open dense subset of $\mathrm{Sp}(n)/ \mathrm{Sp}(1)^n$ and $p(\mathcal{D})$ such that $\Phi \circ \pi =p$.  This gives rise to a local set of coordinates  on $F_{1,2,\dots ,n-1}(\mathbb H^n)$. Those coordinates are compatible with the metric in the sense that $p$ is a Riemannian submersion, which implies that $\Phi$ is an isometry. Those coordinates are only implicit, so we will rather work with the parametrization
\begin{align}\label{local affine parametrization}
w_i:=\left(\begin{matrix}
        q_{1i}q^{-1}_{ni} \\
        \vdots \\
        q_{(n-1)i}q^{-1}_{ni}
        \end{matrix}
        \right), \quad (q_{ij})_{1\leq i,j\leq n} \in \mathrm{Sp}(n).
\end{align}

Notice that the domain $\mathcal{O}$ of those variables can explicitly be described as
\begin{align}\label{eq-O}
    \mathcal{O}=\left\{ w=(w_1,\dots,w_{n-1},w_n) \in \mathbb{H}^{n-1} \times \dots \times \mathbb{H}^{n-1}  \mid w_i^* w_j =-1,\,  1 \le i <j \le n \right\},
\end{align}
which yields a nice parametrization of a dense open subset of $F_{1,2,\dots ,n-1}(\mathbb H^n)$ by the algebraic manifold $\mathcal{O}$. This parametrization will be extensively used in the sequel. 

\subsection{The quaternionic full flag Brownian motion and its radial part}

\begin{definition}
The symplectic  Brownian motion $U(t)=(q_{ij}(t))_{1 \le i,j \le n}$ is the diffusion process on $\mathrm{Sp}(n)$ with generator $\frac{1}{2} \Delta_{\mathrm{Sp}(n)}$ where $\Delta_{\mathrm{Sp}(n)}$ is the Laplacian on $\mathrm{Sp}(n)$ given by \eqref{eq:Delta-left-sum}. 
\end{definition}

\begin{remark}
  Since  the Riemannian metric on $\mathrm{Sp}(n)$ is bi-invariant, a symplectic Brownian motion is both left and right invariant meaning that for every $g \in \mathrm{Sp}(n)$ we have the equality in distribution 
  \[
  (gU(t)g^{-1})_{t \ge 0}\stackrel{\mathclap{\mathrm{d}}}{=} (U(t))_{t \ge 0}.
  \]
\end{remark}

Although we will not make use of it in this paper, we point out that since $\mathrm{Sp}(n)$ is a Lie group, $(U(t))_{t\geq 0}$ can be represented as a solution of the stochastic differential equation in Stratonovitch form
\[
dU(t) =U(t) \circ dA(t),
\]
where $(A(t))_{t\geq 0}$ is a Brownian motion on the Lie algebra $\mathfrak{sp}(n)$, see \cite[Chapter 2]{Baudoin2024-is}.







\

We have the following lemma.

\begin{lemma}\label{polarity}
Let $\mathcal{D}$ be given as in \eqref{eq:quaternionic-domain}.
If $U(0) \in \mathcal{D}$, then
\[
\mathbb{P}( \forall t \ge 0, U(t) \in \mathcal{D})=1.
\]
\end{lemma}

\begin{proof}
We first note that the process $\xi(t)=(q_{n1}(t),\dots,q_{nn}(t))$ is a Brownian motion on the $4n-1$ dimensional sphere $\mathbb{S}^{4n-1}$. This is because the map $\mathrm{Sp}(n)\to \mathbb{S}^{4n-1}$ that associates a symplectic matrix to its last row is a Riemannian submersion with totally geodesic fibers.
Define
\begin{align*}
    \tilde{\mathcal{D}} :=\{ (q_{1},\dots ,q_n)\in\mathbb{S}^{4n-1}\mid q_1\neq 0,\dots ,q_n\neq 0\} .
\end{align*}
It is then enough to prove that the set $\mathbb{S}^{4n-1} \setminus \tilde{\mathcal D}$ is polar for the Brownian motion $\xi$, but $\mathbb{S}^{4n-1} \setminus \tilde{\mathcal D}$ can be written as
\[
\mathbb{S}^{4n-1} \setminus \tilde{\mathcal D}=\mathcal{O}_1 \cup \cdots \cup \mathcal{O}_n ,
\]
where $\mathcal{O}_i=\left\{ q=(q_1,\dots,q_n) \in \mathbb{S}^{4n-1} \mid q_i=0 \right\}$. Now, each of the $\mathcal{O}_i$'s is a submanifold  of dimension $4n-5$, which is strictly smaller than $4n-1$, the dimension of $\mathbb{S}^{4n-1}$. Therefore $\mathcal{O}_i$ is a polar set for $\xi$, which immediately implies that $\mathbb{S}^{4n-1} \setminus \tilde{\mathcal D}$ is a polar set for $\xi$.
\end{proof}

\begin{definition}
The Brownian motion on the quaternionic full flag manifold is the diffusion on $F_{1,2\dots ,n-1}(\mathbb{H}^n)$ with generator $\frac{1}{2} \Delta_{F_{1,2\dots ,n-1}(\mathbb{H}^n)}$. 
\end{definition}

Since the projection map $\mathrm{Sp}(n) \to \mathrm{Sp}(n)/\mathrm{Sp}(1)^n$ is a Riemannian submersion with totally geodesic fibers, it follows that the Brownian motion on $\mathrm{Sp}(n)$ projects down to the Brownian motion on $ \mathrm{Sp}(n)/\mathrm{Sp}(1)^n$. Therefore, from the discussion in Section \ref{section full flag} and Lemma \ref{polarity}, we deduce that if $U(t)=(q_{ij}(t))_{1\le i,j \le n}$ is a Brownian motion on $\mathrm{Sp}(n)$ such that $U(0)\in\mathcal{D}$, then 

\begin{align}\label{BM on flag}
w_i(t)=\left(\begin{matrix}
        q_{1i}(t)q^{-1}_{ni}(t) \\
        \vdots \\
        q_{(n-1)i}(t)q^{-1}_{ni}(t)
        \end{matrix}
        \right), \quad i=1,\dots,n,
\end{align}
parametrizes a Brownian motion on the quaternionic full flag. The \emph{squared radial process}\footnote{This terminology is inspired by Corollary \ref{cor:BM-symplectic-group-in-terms-of-w}.} is defined by
\begin{align}\label{radial variables}
\lambda_j(t)=\frac{1}{1+r_j(t)^2},\quad 1\leq j\leq n,
\end{align}
where
\begin{align*}
r_j(t):=\sqrt{\sum_{i=1}^{n-1} |w_{ij}(t)|^2}.
\end{align*}

\begin{theorem}\label{thm:generator-quaternionic-radial-process}
The process $(\lambda (t))_{t\geq 0}$ is a diffusion with generator $\frac{1}{2}\mathcal{G}$ where
\begin{align}\label{generator lambda}
    \mathcal{G} =4\sum_{j=1}^n\lambda_j(1-\lambda_j )\frac{\partial^2}{\partial\lambda_j^2} 
    -4\sum_{1\leq j\neq\ell\leq n}\lambda_j\lambda_{\ell}\frac{\partial^2}{\partial\lambda_j\partial\lambda_{\ell}}
    +4\sum_{j=1}^n(2-2n\lambda_j )\frac{\partial}{\partial\lambda_j } .
\end{align}
In particular, $(\lambda (t))_{t\geq 0}$ is a Jacobi process on the simplex of index $(3/2,\dots ,3/2)$, see Section \ref{sec: characteristic area} for the definition Jacobi processes on simplices.
\end{theorem}

\begin{proof}
From \eqref{radial variables} we know that
\begin{align*}
\lambda_j(t)=\frac{1}{1+r_j(t)^2}=|q_{nj}(t)|^2 ,
\end{align*}
where $(q_{nj}(t))_{1\le j\le n}$ is the last row of a Brownian motion $(q_{ij}(t))_{1 \le i,j \le n}$ on $\mathrm{Sp}(n)$. Moreover, since the projection of $\mathrm{Sp}(n)$ onto its last row $\mathrm{Sp}(n)\rightarrow\mathbb{S}^{4n-1}$ is a totally geodesic submersion, the process $(q_{n1}(t),\dots, q_{nn}(t))_{t\geq 0}$ is  a Brownian motion on the sphere $\mathbb{S}^{4n-1}$.   

Let  $(\xi(t))_{t\geq 0}$ on the sphere $\mathbb{S}^{4n-1}$ with real components $(\xi_i(t))_{t\geq 0}$, $1 \le i \le 4n$. We then have in distribution
\[
\lambda_j(t)=\sum_{a=0}^3 \xi_{4j-a}(t)^2, \quad t\ge0,\ j=1,\dots, n.
\]

We now show that the generator of $(\lambda (t))_{t\geq 0}$ is given by one half times \eqref{generator lambda}. From \cite[Example 2.4.3]{Baudoin2024-is}, the Laplacian on $\mathbb{S}^{4n-1}$ can be written as
\[
\sum_{j=1}^{n}\sum_{a=0}^3  V_{4j-a}^2
\]
where
\[
V_i=\frac{\partial}{\partial x_i} - x_i \left(\sum_{\ell =1}^{4n} x_{\ell} \frac{\partial}{\partial x_{\ell}} \ \right)
\]
is the orthogonal projection of $\frac{\partial}{\partial x_i}$ onto $T_x\mathbb{S}^{4n-1}$.
Now,
\[
\lambda_j =x_{4j-3}^2 +x_{4j-2}^2+x_{4j-1}^2+x_{4j}^2
\]
so that on functions depending only on $\lambda_1,\dots,\lambda_n$ we have
\[
\frac{\partial}{\partial x_{4j-a} }=2 x_{4j-a} \frac{\partial}{\partial \lambda_j}.
\]
This gives
\[
V_{4j-a}=2x_{4j-a} \left( \frac{\partial}{\partial \lambda_j} -\sum_{\ell =1}^{n} \lambda_{\ell} \frac{\partial}{\partial \lambda_{\ell}} \right)
\]
and therefore $(\lambda(t))_{t\geq 0}$ is a diffusion with generator
 \begin{align*}
 \frac{1}{2}\mathcal{G}=2 \sum_{j=1}^n \lambda_j \left( \frac{\partial}{\partial \lambda_j} -\sum_{\ell =1}^n \lambda_{\ell} \frac{\partial}{\partial \lambda_{\ell}} \right)^2 + \sum_{j=1}^n (4-\lambda_j) \left( \frac{\partial}{\partial \lambda_j} -\sum_{\ell =1}^n \lambda_{\ell} \frac{\partial}{\partial \lambda_{\ell}} \right) .
 \end{align*}
The result follows after a straightforward computation.
\end{proof}

%
%
%
\section{Quaternionic stochastic area processes and skew-product decompositions}\label{sec:quaternionic-area-processes}
In this section, we introduce the quaternionic stochastic area functionals associated with a Brownian motion on the quaternionic full flag manifold.
We then derive explicit expressions for their characteristic functions and prove that these functionals converge in distribution to a multivariate normal distribution.

\subsection{The quaternionic full flag manifold as a spheroid bundle}
The stochastic area processes will be constructed as stochastic line integrals of a natural stochastic area form on the quaternionic full flag manifold. This form is closely related to the connection form of the fibration 
\begin{align}\label{eq:quaternionic-full-flag-fibration}
    \mathrm{Sp}(1)^n\rightarrow\mathrm{Sp}(n)\rightarrow F_{1,2,\dots ,n-1}(\mathbb{H}^n),
\end{align}
which identifies the compact symplectic group as a spheroid\footnote{A \emph{spheroid} is a group of the form $\mathrm{Sp}(1)^n$.} bundle over the quaternionic full flag manifold. This construction is consistent with the general framework developed in \cite[section 3.5.1]{Baudoin2024-is}.
We introduce the notations
\[
    dq_j := dt_j + dx_j \mathbf{i} + dy_j \mathbf{j} + dz_j \mathbf{k}
    \qquad\text{and}\qquad
    d\overline{q}_j := dt_j - dx_j \mathbf{i} - dy_j \mathbf{j} - dz_j \mathbf{k}.
\]
These notations interact naturally with quaternionic multiplication. In particular, one easily verifies 
\begin{align}\label{leibniz rule quaternionic differential}
    d(q_j q_k) = (dq_j)\, q_k + q_j\, (dq_k),
\end{align}
which will be used throughout the section.

We first determine the connection form on the fibration \eqref{eq:quaternionic-full-flag-fibration}.
\begin{proposition}
    Let $\eta =(\eta_1 ,\dots ,\eta_n)$ be the $\mathfrak{sp}(1)^n$-valued one-form given by
    \begin{align}\label{eq-eta}
        \eta_j :=\frac{1}{2}\sum_{k=1}^n(\overline{q}_{kj}dq_{kj} -d\overline{q}_{kj}q_{kj}) .
    \end{align}
    Then, $\eta$ is the connection form of the fibration \eqref{eq:quaternionic-full-flag-fibration}.
\end{proposition}
\begin{proof}
    We will use the isometrical embedding $\iota :F_{1,2,\dots ,n-1}(\mathbb{H}^n)\hookrightarrow (\mathbb{H}P^{n-1})^n$ and the fact that
    \begin{align*}
        \frac{1}{2}\sum_{k=1}^n(\overline{q}_{k}dq_{k} -d\overline{q}_{k}q_{k})
    \end{align*}
    is the connection form of the quaternionic Hopf fibration
    \begin{align}\label{eq:quaternionic-Hopf-fibtration}
        \mathrm{Sp}(1)\rightarrow \mathbb{S}^{4n-1}\rightarrow\mathbb{H}P^{n-1},
    \end{align}
    see \cite[Theorem 5.1.8.]{Baudoin2024-is}.
    The proposition now follows directly from the fact that the fibration \eqref{eq:quaternionic-full-flag-fibration} is the pullback of the tensored fibration \eqref{eq:quaternionic-Hopf-fibtration} under the inclusion $\iota$.
\end{proof}

Next we introduce the quaternionic area form on the full flag manifold. 
\begin{definition}\label{def-a}
Let $\{{w}_{ij}\}_{1\le i\le n-1, 1\le j\le n}$ be the quaternionic local affine coordinate on the full flag manifold $F_{1,2,\dots ,n-1}(\mathbb{H}^n)$ as given in \eqref{local affine parametrization}. The \emph{quaternionic area form} on $F_{1,2,\dots ,n-1}(\mathbb{H}^n)$ is defined as the $\mathfrak{sp}(1)^n$-valued one form  $\mathfrak{a} =(\mathfrak{a}_1,\dots ,\mathfrak{a}_n)$ with
\begin{align}\label{eq-A}
    \mathfrak{a}_j :=\frac{1}{2}\sum_{k=1}^{n-1}\frac{\overline{w}_{kj}dw_{kj} -d\overline{w}_{kj}w_{kj}}{1+|w_j|^2} .
\end{align}
\end{definition}
\begin{remark}
A straightforward calculation shows that the form associated to $\eta$ in the sense of \cite[Equation (3.5.2.)]{Baudoin2024-is} is precisely given by $\mathfrak{a} $, which motivates the terminology \emph{area form}.
\end{remark}
If we denote by $\Theta:=(\Theta_1,\dots, \Theta_n)$ a point in the spheroid $\mathrm{Sp}(1)^{n}$, the following map provides a local \emph{trivialization} of the fibration \eqref{eq:quaternionic-full-flag-fibration}. Let $\mathcal{D}$ and $\mathcal{O}$ be as given in \eqref{eq:quaternionic-domain} and \eqref{eq-O}  respectively. Define
\begin{align}\label{eq:trivialisation-quaternionic-full-flag}
\begin{split}
  \Psi:  \mathrm{Sp}(1)^{n}\times\mathcal{O}&\longrightarrow\qquad\qquad\qquad\mathcal{D} \\
    (\Theta\ ,\   w)&\longmapsto\begin{pmatrix}
        \frac{w_{11}\Theta_1}{\sqrt{1+|w_1|^2}} & \dots & \frac{w_{1n}\Theta_n}{\sqrt{1+|w_n|^2}} \\
        \vdots & \ddots & \vdots \\
        \frac{w_{(n-1)1}\Theta_1}{\sqrt{1+|w_1|^2}} & \dots & \frac{w_{(n-1)n}\Theta_n}{\sqrt{1+|w_n|^2}} \\
        \frac{\Theta_1}{\sqrt{1+|w_1|^2}} & \dots & \frac{\Theta_n}{\sqrt{1+|w_n|^2}}
    \end{pmatrix} .
\end{split}
\end{align}


Computing the connection form $\eta$ as given in \eqref{eq-eta} using the above trivialization then allows us to relate it to the stochastic area form defined in \eqref{eq-A}. In the following theorem we denote  by $\mathrm{Ad}_{g^{-1}}$ the natural $\mathrm{Ad}$-action of $\mathrm{Sp}(1)^n$ on the spheroid bundle $\mathrm{Sp}(n)$, i.e. 
\[
\mathrm{Ad}_{g^{-1}} (\eta_i (g)) :=\mathrm{Ad}_{\Theta_i^{-1}} (\eta_i (g)) ,
\]
where $g=\Psi(\Theta, w)$. 

\begin{theorem}
Let $\omega=(\omega_1,\dots,\omega_n)$ be the Maurer-Cartan form on $\mathrm{Sp}(1)^n$. Then, for every $g \in \mathrm{Sp}(n)$,
\begin{align}\label{connection form index form}
\eta_i (g)= \omega_i (g) +\mathrm{Ad}_{g^{-1}} ( \pi^* \mathfrak{a}_i(g)),
\end{align}
where $\pi^* \mathfrak{a}$ is the pull-back to $\mathrm{Sp}(n)$ of the quaternionic area form $\mathfrak{a}$ by the Riemannian submersion $\pi: \mathrm{Sp}(n)\to F_{1,2,\dots ,n-1}(\mathbb{H}^n)$. Therefore, in the trivialization  \eqref{eq:trivialisation-quaternionic-full-flag}, the connection form is given by
\begin{align}\label{connection form index form in coordinates}
\eta_i (\Psi (\Theta , w)) =  \Theta_i^{-1} d\Theta_i +\frac{1}{2} \Theta_i^{-1} \left( \sum_{k=1}^{n-1}\frac{\overline{w}_{ki}dw_{ki} -d\overline{w}_{ki}w_{ki}}{1+|w_i|^2}\right)\Theta_i .
\end{align}
\end{theorem}

\begin{proof}
We will first prove \eqref{connection form index form in coordinates}.
Recall that  $q_{ij} =\frac{w_{ij}\Theta_j}{\sqrt{1+|w_j|^2}}$ if $1 \le i \le n-1$, $1 \le j \le n$ and $q_{nj} =\frac{\Theta_j}{\sqrt{1+|w_j|^2}}$ for $1\leq j\leq n$. Furthermore, $w_{ij}=q_{ij}q_{nj}^{-1}$ if $1 \le i \le n-1$ and
\[
dw_{ij}=d(q_{ij}q_{nj}^{-1})=(dq_{ij})q_{nj}^{-1}+q_{ij} d(q_{nj}^{-1})=(dq_{ij})q_{nj}^{-1}-q_{ij} q_{nj}^{-1} (dq_{nj}) q_{nj}^{-1}.
\]
Here we used $d(q_{nj}^{-1})=-q_{nj}^{-1}d(q_{nj})q_{nj}^{-1}$ which follows from \eqref{leibniz rule quaternionic differential} since
\[
0=d(q_{nj}q_{nj}^{-1})=(dq_{nj})q_{nj}^{-1}+q_{nj}d(q_{nj}^{-1}).
\]
This gives
\[
\overline{w}_{ij} dw_{ij}-d\overline{w}_{ij} w_{ij}=\overline{q}^{-1}_{nj}(\overline{q}_{ij} dq_{ij}-d\overline{q}_{ij} q_{ij})q_{nj}^{-1}-|q_{ij}|^2 |q_{nj}|^{-2} \left( (dq_{nj}) q_{nj}^{-1}-\overline{q}_{nj}^{-1}d\overline{q}_{nj}\right).
\]
Using the fact that $\sum_{i=1}^n |q_{ij}|^2=1$ one deduces for any $1\le j\le n$,
\begin{align*}
 \sum_{i=1}^{n-1}(\overline{w}_{ij} dw_{ij} - &d\overline{w}_{ij} w_{ij}) \\
=&\sum_{i=1}^{n-1}\overline{q}^{-1}_{nj}(\overline{q}_{ij} dq_{ij}-d \overline{q}_{ij} q_{ij})q_{nj}^{-1}-\left(\sum_{i=1}^{n-1}|q_{ij}|^2 \right)|q_{nj}|^{-2} \left( (dq_{nj}) q_{nj}^{-1}-\overline{q}_{nj}^{-1}d\overline{q}_{nj}\right) \\
=&\sum_{i=1}^{n}\overline{q}^{-1}_{nj}(\overline{q}_{ij} dq_{ij}-d\overline{q}_{ij} q_{ij})q_{nj}^{-1}-|q_{nj}|^{-2} \left( (dq_{nj}) q_{nj}^{-1}-\overline{q}_{nj}^{-1}d\overline{q}_{nj}\right).
\end{align*}
Since $q_{nj} =\frac{\Theta_j}{\sqrt{1+|w_j|^2}}$ and therefore $|q_{nj}|^2=\frac{1}{1+|w_j|^2}$, this yields 
\begin{align}\label{eq-mid1}
 \frac{1}{2}\sum_{i=1}^{n-1}\frac{\overline{w}_{ij} dw_{ij}-d\overline{w}_{ij} w_{ij}}{1+|w_j|^2} &= \frac{1}{2}\sum_{i=1}^{n}\Theta_j (\overline{q}_{ij} dq_{ij}-d\overline{q}_{ij} q_{ij})\Theta_j^{-1}-\frac{1}{2} \left( (dq_{nj}) q_{nj}^{-1}-\overline{q}_{nj}^{-1}d\overline{q}_{nj}\right)\notag \\
&=\Theta_j \eta _j \Theta_j^{-1}-\frac{1}{2} \left( (dq_{nj}) q_{nj}^{-1}-\overline{q}_{nj}^{-1}d\overline{q}_{nj}\right) 
\end{align}
for any $1\le j\le n$.
Finally, we compute
\begin{align*}
d q_{nj} &= d\left( \frac{\Theta_j}{\sqrt{1+|w_j|^2}} \right) =\frac{d\Theta_j}{\sqrt{1+|w_j|^2}}+\Theta_j d\left( \frac{1}{\sqrt{1+|w_j|^2}} \right)
\end{align*}
so that
\begin{align}\label{eq-mid2}
\frac{1}{2} \left( (dq_{nj}) q_{nj}^{-1}-\overline{q}_{nj}^{-1}d\overline{q}_{nj}\right)=\frac{1}{2} \left( (d\Theta_j)\Theta_j^{-1}-\overline{\Theta}_j^{-1} d\overline{\Theta}_j \right)=(d\Theta_j)\Theta_j^{-1}.
\end{align}
Combining \eqref{eq-mid1} and \eqref{eq-mid2} we obtain \eqref{connection form index form in coordinates}. Lastly \eqref{connection form index form} follows from the fact that $\omega_i(\Psi (\Theta , w))=\Theta_i^{-1} d\Theta_i$ and that
\[
\pi^* \mathfrak{a}_i(\Psi (\Theta , w))=\sum_{k=1}^{n-1}\frac{\overline{w}_{ki}dw_{ki} -d\overline{w}_{ki}w_{ki}}{1+|w_i|^2}
\]
because $\pi (\Psi (\Theta , w))$ is precisely the point in $F_{1,2,\dots ,n-1}(\mathbb{H}^n)$ parametrized by \eqref{local affine parametrization}. 

\end{proof}

\subsection{Horizontal Brownian motion on \texorpdfstring{$\mathrm{Sp}(n)$}{Sp(n)} and stochastic areas}
In this section we construct the stochastic area process using the stochastic line integral of the area form introduced above. We also construct the horizontal Brownian motion on $\mathrm{Sp}(n)$ as the horizontal lift of the Brownian motion on $F_{1,2,\dots ,n-1}(\mathbb{H}^n)$. As we will see, in the trivialization \eqref{eq:trivialisation-quaternionic-full-flag}, the fiber motion of the horizontal Brownian motion can be expressed in terms of the quaternionic stochastic area form. This aligns with the general properties of  Riemannian submersion with totally geodesic fibers, see \cite[Theorem 3.1.10]{Baudoin2024-is}.

\begin{definition}\label{definition stochastic area}
    Let $(w(t))_{t\geq 0}$ be a Brownian motion on $F_{1,2,\dots ,n-1}(\mathbb{H}^n)$ as defined in \eqref{BM on flag}. Let $\mathfrak{a}$ be the area form as given in Definition \ref{def-a}.  The \emph{quaternionic stochastic area processes} are given by
    \begin{align}\label{eq-st-area}
        \mathfrak{a}_j(t):=\int_{w_j[0,t]}\mathfrak{a}_j =\frac{1}{2}\sum_{k=1}^n \int_0^t \frac{\overline{w}_{kj}(s)dw_{kj}(s) -d\overline{w}_{kj}(s)w_{kj}(s)}{1+|w_j(s)|^2},\quad  1\leq j\leq n,
    \end{align}
    where the stochastic line integrals can equivalently be understood in the Stratonovich or Itō sense.
    We write $\mathfrak{a}(t) :=(\mathfrak{a}_1(t),\dots ,\mathfrak{a}_n(t))$.
\end{definition}
\begin{definition}
Let $X_{jk}^{(a)}$, $1\le j<k\le n, a=0,1, 2, 3$ be the left-invariant orthonormal frame of the horizontal bundle $\mathcal{H}$ as given in \eqref{horizontal bundle}. The horizontal Laplacian of the submersion $\mathrm{Sp}(n) \to F_{1,2,\dots ,n-1}(\mathbb{H}^n)$ is then given by
\[
\Delta_\mathcal{H}=\sum_{1\le j<k\le n}\ \sum_{a=0}^3 \big(X_{jk}^{(a)}\big)^{\!2}.
\]
The \emph{horizontal Brownian motion} on $\mathrm{Sp}(n)$ is the diffusion on $\mathrm{Sp}(n)$ with generator $\frac{1}{2} \Delta_\mathcal{H}$.  
\end{definition}

We now show that the horizontal  Brownian motion on $\mathrm{Sp}(n)$ can obtained as the horizontal lift of Brownian motion on $F_{1,2,\dots ,n-1}(\mathbb{H}^n)$, with the fiber motion given by the stochastic area processes.
\begin{theorem}\label{prop:horizontal-BM-symplectic-group}
    Let $(w(t))_{t\ge 0}$ be a Brownian motion on $F_{1,2,\dots ,n-1}(\mathbb{H}^n)$, and let $(\mathfrak{a}(t))_{t\ge 0}$ denote its associated quaternionic stochastic area processes. Define the $\mathrm{Sp}(1)^n$-valued process $(\Theta(t))_{t\ge 0}$ as the solution to the Stratonovich stochastic differential equations
    \begin{align}\label{eq:Theta-process}
        d\Theta_j(t) = -\,\circ d\mathfrak{a}_j(t)\,\Theta_j(t),
        \qquad 1 \le j \le n .
    \end{align}
    
    Then, the $\mathrm{Sp}(n)$-valued diffusion process
    \[
        X(t) :=
        \begin{pmatrix}
            \frac{w_{11}(t)\,\Theta_1(t)}{\sqrt{1+|w_1(t)|^2}} & \dots & \frac{w_{1n}(t)\,\Theta_n(t)}{\sqrt{1+|w_n(t)|^2}} \\
            \vdots & \ddots & \vdots \\
            \frac{w_{(n-1)1}(t)\,\Theta_1(t)}{\sqrt{1+|w_1(t)|^2}} & \dots & \frac{w_{(n-1)n}(t)\,\Theta_n(t)}{\sqrt{1+|w_n(t)|^2}} \\
            \frac{\Theta_1(t)}{\sqrt{1+|w_1(t)|^2}} & \dots & \frac{\Theta_n(t)}{\sqrt{1+|w_n(t)|^2}}
        \end{pmatrix}
    \]
    is a horizontal Brownian motion on $\mathrm{Sp}(n)$.
\end{theorem}
\begin{proof}
One could appeal directly to \cite[Theorem 3.5.2]{Baudoin2024-is}, but there is  a more direct argument that we present below. To prove that $(X(t))_{t\geq 0}$ is a horizontal Brownian motion, it suffices to show that $(X(t))_{t\geq 0}$ is the horizontal lift to $\mathrm{Sp}(n)$ of the process $(w(t))_{t\geq 0}$. Since it is immediate that
\[
p(X(t))=w(t),
\]
it remains to verify that $(X(t))_{t\geq 0}$ is a horizontal process. Since the horizontal bundle is the kernel of the connection form $\eta$, we just need to prove that $\int_{X[0,t]} \eta=0$. From \eqref{connection form index form in coordinates}, we compute
\begin{align*}
\int_{X[0,t]} \eta_i &=\int_{0}^t  \Theta_i^{-1} (s) \circ d\Theta_i (s)+\frac{1}{2}\int_0^t  \Theta_i^{-1}(s) \left( \sum_{k=1}^{n-1}\frac{\overline{w}_{ki}(s)dw_{ki}(s) -(d\overline{w}_{ki}(s))w_{ki}(s)}{1+|w_i(s)|^2}\right)\Theta_i(s)\\
 &=-\int_{0}^t  \Theta_i^{-1} (s) \left( \circ d\mathfrak{a}_i (s)\right) \Theta_i (s) \\
 &\qquad\qquad\qquad\quad +\frac{1}{2}\int_0^t  \Theta_i^{-1}(s) \left( \sum_{k=1}^{n-1}\frac{\overline{w}_{ki}(s)dw_{ki}(s) -(d\overline{w}_{ki}(s))w_{ki}(s)}{1+|w_i(s)|^2}\right)\Theta_i(s)
 =0.
\end{align*}
The last equality follows from \eqref{eq-A}. 
\end{proof}

The above skew-product representation of the horizontal Brownian motion also yields a skew-product representation of the symplectic Brownian motion itself.

\begin{corollary}\label{cor:BM-symplectic-group-in-terms-of-w}
    Let $(w(t))_{t\geq 0}$, $(\mathfrak{a}(t))_{t\geq 0}$, and $(\Theta (t))_{t\geq 0}$ be as given in Theorem \ref{prop:horizontal-BM-symplectic-group}. Let $(\beta (t))_{t\geq 0}$ a left-invariant\footnote{The metric on $\mathrm{Sp}(1)^n$ is bi-invariant, so a left-invariant Brownian is also a right-invariant one.} Brownian motion on $\mathrm{Sp}(1)^n$ independent of $(w(t))_{t\geq 0}$ and $(\Theta (t))_{t\geq 0}$.  Then, the   process
    \begin{align*}
        U(t) :=\begin{pmatrix}
        \frac{w_{11}(t)\Theta_1 (t)\beta_1(t)}{\sqrt{1+|w_1(t)|^2}} & \dots & \frac{w_{1n}(t)\Theta_n(t)\beta_n(t)}{\sqrt{1+|w_n(t)|^2}} \\
        \vdots & \ddots & \vdots \\
        \frac{w_{(n-1)1}(t)\Theta_1 (t)\beta_1(t)}{\sqrt{1+|w_1(t)|^2}} & \dots & \frac{w_{(n-1)n}(t)\Theta_n(t)\beta_n(t)}{\sqrt{1+|w_n(t)|^2}} \\
        \frac{\Theta_1(t)\beta_1(t)}{\sqrt{1+|w_1(t)|^2}} & \dots & \frac{\Theta_n(t)\beta_n(t)}{\sqrt{1+|w_n(t)|^2}}
    \end{pmatrix},\quad t\ge0
    \end{align*}
    is a Brownian motion on $\mathrm{Sp}(n)$. In particular the stochastic line integral
    \begin{align*}
        \eta (t):=\int_{U[0,t]}\eta
    \end{align*}
    is a Brownian motion on $\mathfrak{sp}(1)^n$.
  \end{corollary}
\begin{proof}
From the previous theorem, it is enough to prove that if $(X(t))_{t\geq 0}$ is a horizontal Brownian motion on $\mathrm{Sp}(n)$ and $(\beta_1(t),\dots,\beta_n(t))_{t\geq 0}$ an independent Brownian motion on $\mathrm{Sp}(1)^n$, then the process
\[
U(t)=X(t) \begin{pmatrix}
\beta_1(t) & 0 & \cdots & 0 \\
0 & \beta_2(t) & \cdots & 0 \\
\vdots & \vdots & \ddots & \vdots \\
0 & 0 & \cdots & \beta_n(t)
\end{pmatrix}
\]
is a Brownian motion on $\mathrm{Sp}(n)$. Observe first that the generator of the process $(X(t))_{t\geq 0}$ is
\[
\frac{1}{2}\Delta_\mathcal{H}=\frac{1}{2} \sum_{1\le j<k\le n}\ \sum_{a=0}^3 \big(X_{jk}^{(a)}\big)^{\!2}.
\]
and that the generator of the process 
\[
(\beta(t))_{t \ge 0}:=\left(\begin{pmatrix}
\beta_1(t) & 0 & \cdots & 0 \\
0 & \beta_2(t) & \cdots & 0 \\
\vdots & \vdots & \ddots & \vdots \\
0 & 0 & \cdots & \beta_n(t)
\end{pmatrix}\right)_{t\geq 0}
\]
is
\[
\frac{1}{2}\Delta_\mathcal{V}=\frac{1}{2}  \sum_{j=1}^n \sum_{a=1}^3 \big(H_j^{(a)}\big)^{\!2}.
\]
Then, for every bounded Borel function $f$, one  has
\begin{align*}
    \mathbb{E} \left( f(U(t)) \right)&=\mathbb{E} \left( f(X(t)\beta (t)) \right) \\
     &=\int_{\mathrm{Sp}(n)} \mathbb{E} \left( f(g\beta (t)) \mid X(t)=g \right) d\mathbb{P}_{X(t)} (g) \\
     &=\int_{\mathrm{Sp}(n)} \mathbb{E} \left( f(g\beta (t)) \right) d\mathbb{P}_{X(t)} (g) \\
     &=\int_{\mathrm{Sp}(n)} (e^{\frac{1}{2}t \Delta_\mathcal{V}} f)(g\beta(0)) d\mathbb{P}_{X(t)} (g) \\
     &=e^{\frac{1}{2} t \Delta_\mathcal{H}}  (e^{\frac{1}{2}t \Delta_\mathcal{V}} f) (U(0)) ,
\end{align*}
where in the last step we used the fact that for every $h \in \mathrm{Sp}(1)^n$ the map $ g \to h g h^{-1}$ is an isometry of $\mathrm{Sp}(n)$ which preserves the horizontal bundle, so that for any function $ \phi$, $e^{\frac{1}{2} t \Delta_\mathcal{H}} \phi (h g h^{-1})=e^{\frac{1}{2} t \Delta_\mathcal{H}} \phi (g)$. Then, the two operators $\Delta_\mathcal{H}$ and $\Delta_\mathcal{V}$ commute, which allows to conclude that the generator of $(U(t))_{t\geq 0}$ is
\[
\frac{1}{2}\Delta_\mathcal{H}+\frac{1}{2}\Delta_\mathcal{V}=\frac{1}{2}\Delta_{\mathrm{Sp}(n)} ,
\]
 which proves that $(U(t))_{t\geq 0}$ is a Brownian motion on $\mathrm{Sp}(n)$. 
 
 Finally, using the relation \eqref{connection form index form in coordinates} and the definition \eqref{eq:Theta-process} of $(\Theta (t))_{t\geq 0}$, we obtain
\begin{align*}
\int_{U[0,t]} \eta_i 
=&-\int_{0}^t \beta_i^{-1} (s) \Theta_i^{-1} (s) \left( \circ d \mathfrak{a}_i (s)\right) \Theta_i (s)\beta_i(s) + \int_{0}^t \beta_i^{-1} (s) \circ d\beta_i(s) \\
&+\frac{1}{2}\int_0^t  \beta_i^{-1} (s)\Theta_i^{-1}(s) \left( \sum_{k=1}^{n-1}\frac{\overline{w}_{ki}(s)\circ dw_{ki}(s) -(\circ d\overline{w}_{ki}(s))w_{ki}(s)}{1+|w_i(s)|^2}\right)\Theta_i(s)\beta_i(s)\\
 =&\int_{0}^t \beta_i^{-1} (s) \circ d\beta_i(s) .
\end{align*}
Since  the process 
\[
\left(\int_{0}^t \beta_i^{-1} (s) \circ d\beta_i(s)\right)_{t\geq 0}
\]
is a Brownian motion on $\mathfrak{sp}(1)$, because $(\beta_i(t))_{t\geq 0}$ is a left-invariant Brownian motion on $\mathrm{Sp}(1)$, the conclusion follows.
\end{proof}

\subsection{Joint generator of the radial processes and the stochastic areas}
From the expression \eqref{eq-st-area} of the stochastic area processes, one can observe that, similarly to what occurs in other models of Riemannian submersions (see \cite{Baudoin2024-is}), these processes can be realized as time changed Brownian motions, where the time change is determined by the radial processes squared $ (\lambda(t))_{t\ge0}:=(\lambda_j(t))_{t\ge0, j=1,\dots,n}$ associated with Brownian motion on $F_{1,2,\dots ,n-1}(\mathbb{H}^n)$. For this reason, it is natural to consider the joint process $(\lambda (t),\mathfrak{a}(t))_{t\geq 0}$. This section is devoted to deriving the generator of this joint process.
\begin{theorem}\label{thm:diffusion-quaternionic-area}
     Let $(w(t))_{t\geq 0}$ be a Brownian motion on $F_{1,2,\dots ,n-1}(\mathbb{H}^n)$, $(\mathfrak{a}(t))_{t\geq 0}$ its quaternionic stochastic area processes and let
     \[
     \lambda_j(t)=\frac{1}{1+|w_j(t)|^2},\quad 1\leq j\leq n
     \]
     be as given in \eqref{radial variables}.
      The stochastic process $(\lambda (t),\mathfrak{a}(t))_{t\geq 0}$ is a diffusion with generator
    \begin{align*}
        L:=&2\sum_{j=1}^n\lambda_j(1-\lambda_j )\frac{\partial^2}{\partial\lambda_j^2} 
        -2\sum_{1\leq j\neq\ell\leq n}\lambda_j\lambda_{\ell}\frac{\partial^2}{\partial\lambda_j\partial\lambda_{\ell}}
        +2\sum_{j=1}^n(2-2n\lambda_j )\frac{\partial}{\partial\lambda_j } \\
        &+\frac{1}{2}\sum_{j=1}^n\left(\frac{1}{\lambda_j}-1\right)\left( \sum_{\alpha=1}^3 \frac{\partial^2}{\partial (\mathfrak{a}^\alpha_{j})^2 }\right)
        +\frac{1}{2}\sum_{1\leq j\neq\ell\leq n} \sum_{\alpha=1}^3 \frac{\partial^2}{\partial \mathfrak{a}^\alpha_{j}\partial \mathfrak{a}^\alpha_{\ell} } .
    \end{align*}
\end{theorem}
\begin{proof}
    The quadratic covariations between $\lambda_j$ and $\lambda_{\ell}$ follow directly from Theorem \ref{thm:generator-quaternionic-radial-process}.
    It remains to calculate the quadratic covariations between $\mathfrak{a}_j$ and $\lambda_{\ell}$, as well as between $\mathfrak{a}_j$ and $\mathfrak{a}_{\ell}$ for $1\leq j, \ell\leq n$. We decompose each quaternionic area forms as
    \[
    \mathfrak{a}_i = \mathfrak{a}_i^1 \mathbf{i} + \mathfrak{a}_i^2 \mathbf{j}+\mathfrak{a}_i^3 \mathbf{k}
    \]
    and aim to compute the quadratic covariations
    \[
    \left \langle \int_{w[0,\cdot]} \mathfrak{a}_i^a, \int_{w[0,\cdot]} \mathfrak{a}_j^b \right \rangle (t),\quad {a,b=1,2,3}
    \]
   Since $(w(t))_{t\geq 0}$ is a Brownian motion on $F_{1,2,\dots ,n-1}(\mathbb{H}^n)$, it satisfies the general fact that
    \[
    \left \langle \int_{w[0,\cdot]} \mathfrak{a}_i^a, \int_{w[0,\cdot]} \mathfrak{a}_j^b \right \rangle (t) =\int_0^t g_{F_{n}} (\mathfrak{a}_i^a ,\mathfrak{a}_j^b) (w(s))  ds,
    \]
    where $g_{F_n}$ is the Riemannian cometric on $F_{1,2,\dots n-1}(\mathbb{H}^n)$.
    Thus we are reduced to computing $g_{F_n} (\mathfrak{a}_i^a ,\mathfrak{a}_j^b)$.
    Recall the relation \eqref{connection form index form} between quaternionic area forms and  Maurer-Cartan forms:
   \begin{align*}
\eta_i (g)= \omega_i (g) +\mathrm{Ad}_{g^{-1}} ( \pi^* \mathfrak{a}_i(g)),
\end{align*}
 Since $\pi$ is a Riemannian submersion, $\mathrm{Ad}_{g^{-1}}$ is an isometry, and $\eta$ is purely vertical, we obtain
    \[
    g_{F_n} (\mathfrak{a}_i^a ,\mathfrak{a}_j^b) =g_{\mathrm{Sp}
    (n)} ((\omega_i^a)_\mathcal{H} ,(\omega_j^b)_\mathcal{H}) ,\quad a,b=1,2,3,\ 1\le i,j\le n.
    \]
    where 
    \[
    \omega_i=\omega_i^1 \mathbf{i} + \omega_i^2 \mathbf{j}+\omega_i^3 \mathbf{k}
    \]
    and $(\omega_i^a)_\mathcal{H}$ denotes the horizontal part of the one-form $\omega_i^a$, i.e. the orthogonal projection of $\omega_i^a$ onto the orthogonal complement of $\eta$.
    Recall that the vector fields
 \[
 X^{(a)}_{jk}=\sum_{\ell =1}^n ((q_{\ell j} \mathbf{e}_a) \cdot \partial_{q_{\ell k}}-(q_{\ell k} \overline{\mathbf{e}}_a) \cdot \partial_{q_{\ell j}}),\quad {a=0,1,2, 3,\ 1\le j<k\le n}
 \]   
     form an orthonormal frame of the horizontal bundle.    
    Therefore, for any $\alpha, \alpha'\in\{1,2,3\}$ and $1\le l,r\le n$,
    \[
    g_{\mathrm{Sp} (n)} ((\omega_l^\alpha)_\mathcal{H} ,(\omega_r^{\alpha'})_\mathcal{H})
    =\sum_{a=0}^3 \sum_{1\leq j<k\leq n} g_{\mathrm{Sp}
    (n)} (\omega_l^\alpha,\mathcal{X}^a_{jk})g_{\mathrm{Sp}
    (n)} (\omega_r^{\alpha'},\mathcal{X}^a_{jk}) ,
    \]
    where $\mathcal{X}^a_{jk}$ is the one-form dual of $X_{jk}^{(a)}$, i.e. $ \mathcal{X}^a_{jk}=g_{\mathrm{Sp}(n)}\left( X_{jk}^{(a)},\cdot \right).$ Here we used the expansion
    \begin{align*}
        (\omega_l^\alpha)_\mathcal{H} =\sum_{a =0}^3\sum_{1\leq j<k\leq n}g_{\mathrm{Sp}
    (n)} (\omega_l^\alpha,\mathcal{X}^a_{jk})\mathcal{X}^a_{jk} .
    \end{align*}
    
    We can now perform the computation in real coordinates, using the following  inner product notation: for $h=t +x \mathbf{i} +y \mathbf{j} +z \mathbf{k}$, we write
    \[
    h\cdot dq_{ij}=t dt_{ij}+x dx_{ij}+y dy_{ij} +z dz_{ij}.
    \]
    From \cite[Page 30]{Baudoin2024-is} we have  for any $1\le \alpha \le 3$, $1\le l \le n$ that 
    \[
    \omega_l^\alpha=\frac{1}{|q_{nl}|^2} ( q_{nl} \mathbf{e}_\alpha)\cdot dq_{nl}.
    \]
    Moreover, it is easy to compute that 
    \[
    \mathcal{X}^a_{jk}=\frac{1}{2}\sum_{i=1}^n (( q_{ij} \mathbf{e}_a)\cdot dq_{ik} -( q_{ik}\overline{\mathbf{e}}_a)\cdot dq_{ij}).
    \]
    Finally, recall that $\mathrm{Sp}(n)$ is an immersed submanifold of the Euclidean space $\mathbb{R}^{4 n^2}$, equipped with the inner product equal to $1/2$ of the standard Euclidean inner product.   Therefore, for any $1\leq l<r\leq n$, we compute
    \begin{align*}
    g_{\mathrm{Sp} (n)} ((\omega_l^\alpha)_\mathcal{H} ,(\omega_r^{\alpha'})_\mathcal{H})&=\sum_{a=0}^3 \sum_{1\leq j<k\leq n} g_{\mathrm{Sp}
    (n)} (\omega_l^\alpha,\mathcal{X}^a_{jk})g_{\mathrm{Sp}
    (n)} (\omega_r^{\alpha'},\mathcal{X}^a_{jk}) \\
    &=\sum_{a=0}^3 g_{\mathrm{Sp}
    (n)} (\omega_l^\alpha,\mathcal{X}^a_{lr}) g_{\mathrm{Sp}
    (n)} (\omega_r^{\alpha'},\mathcal{X}^a_{lr}) \\
    &=-\frac{1}{|q_{nl}|^2|q_{nr}|^2}\sum_{a=0}^3 \left( q_{nl}\mathbf{e}_\alpha \cdot q_{nl} \overline{\mathbf{e}}_a\right)\left( q_{nr}\mathbf{e}_{\alpha'} \cdot q_{nr} \mathbf{e}_a\right)  \\
    &=\frac{1}{|q_{nl}|^2|q_{nr}|^2}\sum_{a=1}^3  \left( q_{nl}\mathbf{e}_\alpha \cdot q_{nl} \mathbf{e}_a\right)\left( q_{nr}\mathbf{e}_{\alpha'} \cdot q_{nr} \mathbf{e}_a\right)
    =\delta_{\alpha \alpha'} .
    \end{align*}
    When $1\leq l=r\leq n$, we have
    \begin{align*}
    g_{\mathrm{Sp} (n)} ((\omega_l^\alpha)_\mathcal{H} ,(\omega_l^{\alpha'})_\mathcal{H})
    &=\sum_{a=0}^3 \sum_{1\leq j\neq l\leq n} g_{\mathrm{Sp} (n)} (\omega_l^\alpha,\mathcal{X}^a_{jl})g_{\mathrm{Sp} (n)} (\omega_l^{\alpha'},\mathcal{X}^a_{jl}) \\
    &=\frac{1}{|q_{nl}|^4}\sum_{a=0}^3 \sum_{1\leq j\neq l\leq n}( q_{nl}\mathbf{e}_\alpha)\cdot (q_{nj} \mathbf{e}_a)  (q_{nl}\mathbf{e}_{\alpha'})\cdot (q_{nj} \mathbf{e}_a) \\
    &=\frac{1}{|q_{nl}|^4} \sum_{1\leq j\neq l\leq n}( q_{nl}\mathbf{e}_\alpha)\cdot  (q_{nl}\mathbf{e}_{\alpha'}) | q_{nj}|^2 ,
    \end{align*}
    where in the equality we used that $ q_{nj}\mathbf{e}_a$, $a=0,1,2,3$  forms an orthogonal basis of $\mathbb R^4$. Therefore, if $1\leq \alpha \neq \alpha'\leq 3 $  we have
    \[
    g_{\mathrm{Sp}
    (n)} ((\omega_l^\alpha)_\mathcal{H} ,(\omega_l^{\alpha'})_\mathcal{H}) =0,
    \]
    whereas if $\alpha = \alpha'$ we have
    \[
    g_{\mathrm{Sp}
    (n)} ((\omega_l^\alpha)_\mathcal{H} ,(\omega_l^{\alpha'})_\mathcal{H})=\frac{1-|q_{nl}|^2}{|q_{nl}|^2}=\frac{1-\lambda_\ell}{\lambda_\ell}.
    \]
    In a similar way we prove that $\left\langle \lambda_i , \mathfrak{a}_j \right\rangle_t=0$. Indeed,
  \[
   \left\langle \lambda_i , \mathfrak{a}^a_j \right\rangle_t= \left \langle \int_{w[0,\cdot]} d\lambda_i, \int_{w[0,\cdot]} \mathfrak{a}_j^a \right \rangle (t) =\int_0^t g_{F_{n}} (d\lambda_i ,\mathfrak{a}_j^a) (w(s)) ds.
    \]
Now, observe that the one-form $d\lambda_i =d|q_{ni}|^2=2q_{ni}\cdot dq_{ni}$ is horizontal on $\mathrm{Sp}(n)$. Therefore, as before we deduce that
\[
g_{F_{n}} (d\lambda_i ,\mathfrak{a}2_j^a) =-2 g_{\mathrm{Sp}(n)}(q_{ni}\cdot dq_{ni}, \omega_j^a)=0 ,
\]
because
\[
    \omega_j^a=\frac{1}{|q_{nj}|^2} ( q_{nj} \mathbf{e}_a)\cdot dq_{nj}.
    \]
\end{proof}

\subsection{Characteristic function of the quaternionic stochastic areas}\label{sec: characteristic area}

We now turn to the computation of the characteristic function of the quaternionic stochastic areas. We will use the Yor's transform method \cite[Appendix A9]{Baudoin2024-is}. We first need to introduce some notations related to Jacobi operators on simplices, see  \cite[\textsection 2.2]{baudoin2025fullflag} for further details. Consider the simplex 
\begin{equation*}
 \Sigma_{n-1} :=    \{ \lambda \in \mathbb{R}^{n-1} \mid  \lambda_j \geq 0, \, 1 \leq j \leq n-1, \quad \lambda_1 + \dots + 
    \lambda_{n-1} \leq 1\}. 
\end{equation*}
The Jacobi operator in $\Sigma_{n-1}$ is  defined by
\begin{equation}\label{JacSim}
\mathcal{G}_\kappa :=\sum_{j=1}^{n-1} \lambda_j(1-\lambda_j) \frac{\partial^2}{\partial \lambda_{j}^2} 
 + \sum_{j=1}^{n-1}\left[\left(\kappa_j+\frac{1}{2}\right) 
 - \left(|\kappa| +\frac{n}{2}\right)\lambda_j\right] \frac{\partial}{\partial \lambda_{j}} 
 -\sum_{1 \leq j \neq \ell \leq n-1}\lambda_j\lambda_{\ell}  \frac{\partial^2}{\partial \lambda_{j}\lambda_{\ell}} .   
 \end{equation}
  A diffusion with generator $\mathcal{G}_\kappa$ is called a Jacobi diffusion in the simplex $\Sigma_{n-1}$.  Here, $\kappa = (\kappa_1, \dots, \kappa_n)$ is a parameter set such that $\kappa_j > -1/2$ for any $1 \leq j \leq n$ and  $|\kappa| := \kappa_1 + \dots + \kappa_n$. The operator $\mathcal{G}_\kappa$ is symmetric with respect to the Dirichlet measure on $\Sigma_{n-1}$ whose density is given by: 
\begin{multline}\label{DenDirich}
W^{(\kappa)}(\lambda_1, \dots, \lambda_{n-1}) := \frac{\Gamma(|\kappa|+(n/2))}{\prod_{j=1}^n\Gamma(\kappa_j+(1/2))}  \lambda_1^{\kappa_1-(1/2)}\cdots
\lambda_{n-1}^{\kappa_{n-1}-(1/2)} 
\\ (1-\lambda_1-\cdots-\lambda_{n-1})^{\kappa_n-(1/2)}. 
\end{multline}
The spectrum of $\mathcal{G}_\kappa$ is discrete and is given by 
\begin{equation}\label{eq-eigenv-prelim}
-j\left( j+|\kappa| +\frac{n-2}{2}\right), \quad j \geq 0.
\end{equation}
The corresponding set of orthonormal eigenfunctions consists of the so-called Jacobi polynomials on the simplex, which we will denote by $P_{\tau}^{(\kappa)}$ for $\tau\in\mathbb{N}^{n-1}$.
We refer the reader to the monograph \cite{MR3289583}, 
and the papers \cite{aktacs2013sobolev} and \cite{MR2817619} for further details about these polynomials. The density with respect to the Dirichlet measure $W^{(\kappa)}$ of the heat semi-group $e^{t\mathcal{G}_\kappa}$ reads 

\begin{equation}\label{kernel jacobi simplex}
q^{(\kappa_1,\dots ,\kappa_{n-1},\kappa_n)}_t(x,y) =\sum_{\tau\in\mathbb{N}^{n-1}} 
e^{-|\tau |(|\tau |+|\kappa |+(n-2)/2)t} P_{\tau}^{(\kappa)}(x)P_{\tau}^{(\kappa)}(y).  
\end{equation}

We can lift Jacobi diffusions in $\Sigma_{n-1}$ to diffusions in the $n-1$ simplex of $\mathbb{R}^n$.  More precisely, define
\begin{align}\label{eq-simplex}
    \mathcal{T}_n :=\{\lambda\in\mathbb{R}^{n}\mid\lambda_j\geq 0, \, 1\leq j\leq n,\, \lambda_1 +\dots +\lambda_n =1\}.
\end{align}
It is easy to check that if $(\lambda_1(t),\dots,\lambda_{n-1}(t))_{t\geq 0}$ is a diffusion with generator $\mathcal{G}_\kappa$, then the process $(\lambda_1(t),\dots,\lambda_{n}(t))_{t\geq 0}$, where $\lambda_n(t) :=1-\sum_{k=1}^{n-1} \lambda_k(t)$, is a diffusion in $\mathcal{T}_n$ with generator
\begin{equation}\label{GenJacSim1}
\widehat{\mathcal{G}}_\kappa:=\sum_{j=1}^{n} \lambda_j(1-\lambda_j) \frac{\partial^2}{\partial \lambda_{j}^2} 
 + \sum_{j=1}^{n}\left[\left(\kappa_j+\frac{1}{2}\right) 
 - \left(|\kappa| +\frac{n}{2}\right)\lambda_j\right] \frac{\partial}{\partial \lambda_{j}} 
 -\sum_{1 \leq j \neq \ell \leq n}\lambda_j\lambda_{\ell}  \frac{\partial^2}{\partial \lambda_{j}\partial \lambda_{\ell}} . 
 \end{equation}

The operator $\widehat{\mathcal{G}}_\kappa$ is called the lift of $\mathcal{G}_\kappa$ to $\mathcal{T}_n$ and a diffusion with generator $\widehat{\mathcal{G}}_\kappa$ is called a Jacobi diffusion in $ \mathcal{T}_n$.

\begin{theorem}\label{thm:characteristic-function-quaternionic-area}
    Let $(\lambda (t),\mathfrak{a}(t))_{t\geq 0}$ be the diffusion process from Theorem \ref{thm:diffusion-quaternionic-area}.
    For any $u =(u_1^1,u_1^2,u_1^3,\dots ,u_n^1,u_n^2,u_n^3)\in\mathbb{R}^{3n}$, any $\lambda (0),\lambda$ in the interior of $\mathcal{T}_{n}$ and any $t>0$ we have
    \begin{align*}
        \mathbb{E}\bigg( &e^{i\sum_{j=1}^n\sum_{a=1}^3 u_j^a\mathfrak{a}^a_j(t)}\biggm|\lambda (t)=\lambda\bigg)
        =e^{-(2n-2)\sum_{j=1}^n\mu_j t -\frac{1}{2}\sum_{1\leq j\neq\ell\leq n}(\sum_{a=1}^3u_j^au_{\ell}^a +\mu_j\mu_{\ell})t}\\
        &\qquad\prod_{j=1}^n\left(\frac{\lambda_j (0)}{\lambda_j}\right)^{\frac{\mu_j}{2}}\frac{q^{(3/2+\mu_1 ,\dots ,3/2+\mu_{n})}_{2t} (\lambda^{(n-1)}(0),\lambda^{(n-1)})}{q^{(3/2 ,\dots ,3/2)}_{2t} (\lambda^{(n-1)}(0),\lambda^{(n-1)})}\frac{W^{(3/2 +\mu_1,\dots ,3/2 +\mu_n)}(\lambda^{(n-1)})}{W^{(3/2,\dots ,3/2)}(\lambda^{(n-1)})} ,
    \end{align*}
    where $q_t^{(\kappa_1 ,\dots ,\kappa_n )}$ is the heat kernel of a Jacobi process on the simplex of index $(\kappa_1,\dots ,\kappa_n)$ and $W^{(\kappa_1,\dots ,\kappa_1)}$ its invariant measure density. 
    Furthermore $\mu_j :=\sqrt{|u_j^1|^2+|u_j^2|^2 +|u_j^3|^2 +1}-1$ and $\lambda^{(n-1)}:= (\lambda_1,\dots ,\lambda_{n-1})$.
\end{theorem}
\begin{proof}
    From Theorem \ref{thm:diffusion-quaternionic-area} we know that, conditioned on $(\lambda (s),s\leq t)$, the quaternionic stochastic area process $(\mathfrak{a}(t))_{t\geq 0}$ is a Gaussian variable with mean zero and covariance matrix
    \begin{align*}
        \Sigma (t):=\begin{pmatrix}
            \int_0^t\frac{1-\lambda_1 (s)}{\lambda_1 (s)}ds & \dots & t\\
            \vdots & \ddots & \vdots \\
            t & \dots & \int_0^t\frac{1-\lambda_n (s)}{\lambda_n (s)}ds
        \end{pmatrix}\otimes I_3 ,
    \end{align*}
    where $\otimes$ stands for the Kronecker product of matrices.
    It follows that
    \begin{align*}
        \mathbb{E}\bigg( &e^{i\sum_{j=1}^n\sum_{a=1}^3 u_j^{a}\mathfrak{a}_j^a (t)}\mid\lambda (t)=\lambda\bigg)
        =\mathbb{E}(e^{-\frac{1}{2}u^T\Sigma (t) u}\mid\lambda (t)=\lambda)\\
        &=\mathbb{E}\left(\exp\left( -\frac{1}{2}\sum_{j=1}^n\sum_{a=1}^3(u_j^{a})^2\int_0^t\frac{1-\lambda_j (s)}{\lambda_j (s)}ds -\frac{1}{2}\sum_{1\leq j\neq\ell\leq n}\sum_{a=1}^3u_j^au_{\ell}^a t\right)\bigg|\lambda (t)=\lambda\right) .
    \end{align*}
    
    It only remains to derive an expression for
    \begin{align*}
        \mathbb{E}\left(\exp\left( -\frac{1}{2}\sum_{j=1}^n\sum_{a=1}^3(u_j^{a})^2\int_0^t\frac{1-\lambda_j (s)}{\lambda_j (s)}ds\right)\bigg|\lambda (t)=\lambda\right) .
    \end{align*}
    Define the following function on $\mathcal{T}_{n}$
    \begin{align*}
        f(\lambda_1 ,\dots ,\lambda_n ):=\prod_{j=1}^n\lambda_j^{\frac{\mu_j}{2}}\textrm{ for some }\mu_j >0 .
    \end{align*}
    Let $\hat{\mathcal{G}}_{3/2,\dots ,3/2}$ be a Jacobi operator on $\mathcal{T}_n$ of index $(3/2,\dots ,3/2)$.
    By applying $\hat{\mathcal{G}}_{3/2,\dots ,3/2}$ to $f$ we obtain
    \begin{align*}
        \hat{\mathcal{G}}_{3/2,\dots ,3/2} f
        &=\sum_{j=1}^n\lambda_j(1-\lambda_j)\frac{\partial^2 f}{\partial\lambda_j^2} +\sum_{j=1}^n (2-2n\lambda_j)\frac{\partial f}{\partial\lambda_j} -\sum_{1\leq j\neq\ell\leq n}\lambda_j\lambda_{\ell}\frac{\partial^2 f}{\partial\lambda_j\partial\lambda_{\ell}}\\
        &=\sum_{j=1}^n\frac{\mu_j}{2}\left(\frac{\mu_j}{2} -1\right)\frac{1-\lambda_j}{\lambda_j}f +\sum_{j=1}^n\mu_j\left(\frac{1}{\lambda_j} -n\right) f-\sum_{1\leq j\neq\ell\leq n}\frac{\mu_j\mu_{\ell}}{4}f ,
    \end{align*}
    which implies that $f$ is an eigenfunction of the operator
    \begin{align*}
        \hat{\mathcal{G}}_{3/2,\dots ,3/2} -\frac{1}{2}\sum_{j=1}^n\left(\frac{\mu_j^2}{2} +\mu_j\right)\frac{1-\lambda_j}{\lambda_j} 
    \end{align*}
    associated with the eigenvalue
    \begin{align*}
        -\left( n-1\right)\sum_{j=1}^n\mu_j -\frac{1}{4}\sum_{1\leq j\neq\ell\leq n}\mu_j\mu_{\ell} .
    \end{align*}
    We choose $\mu_j =\sqrt{1+|u_j^1|^2+|u_j^2|^2 +|u_j^3|^2} -1\geq 0$, since in this case $f$ is an eigenfunction of the operator
    \begin{align*}
        \hat{\mathcal{G}}_{3/2,\dots ,3/2} -\frac{1}{4}\sum_{j=1}^n\sum_{a=1}^3 (u_j^a)^2\frac{1-\lambda_j}{\lambda_j} . 
    \end{align*}
    
    Now we recall that $(\lambda (t))_{t\geq 0}$ is a Jacobi process in the simplex $\mathcal{T}_n$ with generator $2\hat{\mathcal{G}}_{3/2,\dots ,3/2}$.
    Itō's formula then shows that the process
    \begin{align*}
        D_t^u :=e^{\left( 2n-2\right)\sum_{j=1}^n\mu_jt +\frac{1}{2}\sum_{1\leq j\neq\ell\leq n}\mu_j\mu_{\ell} t}\left(\prod_{j=1}^n\left(\frac{\lambda_j (t)}{\lambda_j (0)}\right)^{\frac{\mu_j}{2}} e^{-\sum_{a=1}^3\frac{1}{2}(u_j^a)^2\int_0^t\frac{1-\lambda_j (s)}{\lambda_j (s)}ds}\right)
    \end{align*}
    is a local martingale.
    One can easily verify that
    \begin{align*}
        D_t^u\leq\frac{e^{\left( 2n-2\right)\sum_{j=1}^n\mu_jt +\frac{1}{2}\sum_{1\leq j\neq\ell\leq n}\mu_j\mu_{\ell} t}}{\prod_{j=1}^n\lambda_j (0)^{\frac{\mu_j}{2}}} ,
    \end{align*}
    which implies that $(D_t^{u})_{t\geq 0}$ is in fact a martingale.
    We can therefore define a new probability measure $\mathbb{P}^u$ by setting $d\mathbb{P}^u :=D_t^ud\mathbb{P}$.
    We then have for every bounded Borel function $F$
    \begin{align*}
        \mathbb{E}\bigg( &F(\lambda_1 (t),\dots ,\lambda_n (t))e^{-\frac{1}{2}\sum_{j=1}^n\sum_{a=1}^3(u_j^a)^2\int_0^t\frac{1-\lambda_j (s)}{\lambda_j (s)}ds}\bigg) \\
        &\qquad =e^{-(2n-2)\sum_{j=1}^n\mu_j t -\frac{1}{2}\sum_{1\leq j\neq\ell\leq n}\mu_j\mu_{\ell} t}\prod_{j=1}^n\lambda_j(0)^{\frac{\mu_j}{2}}\mathbb{E}^u\left(\frac{F(\lambda_1 (t),\dots ,\lambda_n (t))}{\prod_{j=1}^n\lambda_j (t)^{\frac{|u_j|}{2}}}\right) .
    \end{align*}

    Let $s_t^{(u)}(\lambda (0),d\lambda )$ denote the probability distribution of $\lambda (t)$ under $\mathbb{P}^u$ and $\hat{q}_{2t}^{(3/2,\dots ,3/2)}(\lambda (0),d\lambda )$ the probability distribution of $\lambda (t)$ under $\mathbb{P}$.
    The above equality then implies
    \begin{align*}
        \mathbb{E}\bigg( &e^{-\frac{1}{2}\sum_{j=1}^n\sum_{a=1}^3(u_j^a)^2\int_0^t\frac{1-\lambda_j (s)}{\lambda_j (s)}ds}\biggm|\lambda (t)=\lambda\bigg)\hat{q}_{2t}^{(3/2,\dots ,3/2)}(\lambda (0),d\lambda ) \\
        &\qquad =e^{-(2n-2)\sum_{j=1}^n\mu_j t-\frac{1}{2}\sum_{1\leq j\neq\ell\leq n}\mu_j\mu_{\ell} t}\prod_{j=1}^n\left(\frac{\lambda_j(0)}{\lambda_j}\right)^{\frac{\mu_j}{2}} s_t^{(u)}(\lambda (0),d\lambda ) .
    \end{align*}

    We can compute $s_t^{(u)}(\lambda (0),d\lambda )$ using Girsanov's theorem.
    Note that since $(\lambda (t))_{t\geq 0}$ is a Jacobi process on the simplex of index $(3/2,\dots ,3/2)$ it satisfies the stochastic differential equations
    \begin{align*}
        d\lambda_j = 
        2\sum_{\ell =1; \ell\neq j}^n\sqrt{\lambda_{\ell}\lambda_j}d\gamma_{\ell j} +2\left( 2-2n\lambda_j\right) dt,
    \end{align*}
    where $(\gamma_{\ell j}(t))_{\ell <j}$ is a Brownian motion on $\mathbb{R}^{\frac{1}{2}n(n-1)}$ and $\gamma_{\ell j}:=-\gamma_{j\ell}$ for $1\leq j<\ell\leq n$.
    By Itō's formula
    \begin{align*}
        d\ln (\lambda_j(t)) &=\frac{d\lambda_j (t)}{\lambda_j (t)}
        -\frac{1}{2}\frac{d\lambda_j (t)d\lambda_j (t)}{\lambda_j(t)^2}\\
        &=\frac{2}{\sqrt{\lambda_j (t)}}\sum_{\ell =1;\ell\neq j}^nd\gamma_{\ell j}(t)\sqrt{\lambda_{\ell} (t)}
        +2\left(\frac{2}{\lambda_j(t)} -2n\right) dt
        -\frac{2}{\lambda_j (t)}\sum_{\ell =1;\ell\neq j}^n\lambda_{\ell} (t)dt ,
    \end{align*}
    which, using $1-\lambda_j (t)=\sum_{\ell =1;\ell\neq j}^n\lambda_{\ell} (t)$, simplifies to 
    \begin{align*}
        d\ln (\lambda_j (t))=\frac{2}{\sqrt{\lambda_j (t)}}\sum_{\ell =1;\ell\neq j}^nd\gamma_{\ell j}(t)\sqrt{\lambda_{\ell} (t)} +\frac{2dt}{\lambda_j(t)}-2(2n-1)dt .
    \end{align*}
    This gives
    \begin{align*}
        \prod_{j=1}^n\lambda_j (t)^{\frac{\mu_j}{2}}=&\exp\left(\sum_{1\leq j\neq\ell\leq n}\mu_j\int_0^t\sqrt{\frac{\lambda_{\ell} (s)}{\lambda_j (s)}}d\gamma_{\ell j}(s)\right) \\
        &\exp\left( \sum_{j=1}^n\int_0^t\mu_j\frac{ds}{\lambda_j (s)}-(2n-1)\sum_{k=1}^n\mu_j t\right) .
    \end{align*}
    Define the stochastic processes $(\tilde{\gamma}_{\ell j}(t))_{t\geq 0}$ by
    \begin{align*}
        d\tilde{\gamma}_{\ell j}(t) :=d\gamma_{\ell j}(t) -\Theta_{\ell j}(t)dt
    \end{align*}
    and set $\tilde{\gamma}_{j\ell} (t):=-\tilde{\gamma}_{\ell j}(t)$, where
    \begin{align*}
        \Theta_{\ell j}(t):=\frac{\mu_j}{\sqrt{\lambda_j (t)}}\sqrt{\lambda_{\ell} (t)} -\frac{\mu_{\ell}}{\sqrt{\lambda_{\ell} (t)}}\sqrt{\lambda_j (t)} .
    \end{align*}
    By Girsanov's theorem the process $(\tilde{\gamma}_{\ell j}(t))_{1\leq\ell <j\leq n}$ is a Brownian motion on $\mathbb{R}^{\frac{1}{2}n(n-1)}$ under $\mathbb{P}^u$.
    This means that under $\mathbb{P}^u$, the process $(\lambda (t))_{t\geq 0}$ satisfies the stochastic differential equations
    \begin{align*}
        d\lambda_j (t) =&2\sqrt{\lambda_j (t)}\sum_{\ell =1;\ell\neq j}^n d\tilde{\gamma}_{\ell j}(t)\sqrt{\lambda_{\ell} (t)} +2(2-2n\lambda_j (t))dt\\
        &+2\sum_{1\leq\ell <j}(\mu_j\lambda_{\ell} (t) -\mu_{\ell}\lambda_j (t))dt
        -2\sum_{j<\ell\leq n}(\mu_{\ell}\lambda_j (t) -\mu_j\lambda_{\ell} (t))dt \\
        =&2\sqrt{\lambda_j (t)}\sum_{\ell =1;\ell\neq j}^nd\tilde{\gamma}_{\ell j}\sqrt{\lambda_{\ell} (t)} +2\left( 2-2n\lambda_j (t) -\sum_{k=1}^n\mu_k\lambda_j (t) +\mu_j\right) dt,
    \end{align*}
    where we used that $\sum_{\ell =1}^n\lambda_{\ell}(t) =1$.
    In particular, the generator of $(\lambda (t))_{t\geq 0}$ under $\mathbb{P}^u$ is given by a Jacobi operator on the simplex $\mathcal{T}_n$ of index $(3/2 +\mu_1,\dots ,3/2 +\mu_n)$.
\end{proof}

\subsection{Limit theorem for the quaternionic stochastic areas}

We are now ready to prove a central limit type theorem for the asymptotics of the quaternionic stochastic area.

\begin{theorem}\label{thm:limit-quaternionic-stochastic-area}
    Let $(w(t),\mathfrak{a}(t))_{t\geq 0}$ be the diffusion process from Theorem \ref{thm:diffusion-quaternionic-area}.
    The following convergence holds in distribution
    \begin{align*}
        \frac{\mathfrak{a}(t)}{\sqrt{t}}\rightarrow\mathcal{N}_{3n}(0,\Sigma )\textrm{ as }t\rightarrow\infty ,
    \end{align*}
    where $\mathcal{N}_{3n}(0,\Sigma )$ is a $3n$-dimensional multivariate normal distribution of mean $0$ and covariance $\Sigma$ with
    \begin{align}\label{eq:covariance-limit}
        \Sigma :=\begin{pmatrix}
            2n-2 & 1 &\dots & 1\\
            1 & 2n-2 & \dots & 1\\
            \vdots &\vdots & \ddots & \vdots\\
            1 & 1 & \dots & 2n-2
        \end{pmatrix}\otimes I_3.
    \end{align}
\end{theorem}
\begin{proof}

We give two proofs for this result. The first one is based on the explicit formula for the characteristic function of $\mathfrak{a}$. Let $u_1^1,u_1^2,u_1^3,\dots ,u_n^1,u_n^2,u_n^3\in\mathbb{R}$.
    From Theorem \ref{thm:characteristic-function-quaternionic-area} one has
    \begin{align*}
        \mathbb{E}&\bigg( e^{i\sum_{j=1}^n\sum_{a=1}^3 u_j^a\mathfrak{a}^a_j(t)}\bigg)
        =e^{-(2n-2)\sum_{j=1}^n\mu_j^at -\frac{1}{2}\sum_{1\leq j\neq\ell\leq n}(\sum_{a=1}^3u_j^au_{\ell}^a +\mu_j\mu_{\ell})t}\\
        &\int_{\mathcal{T}_n}\prod_{j=1}^n\left(\frac{\lambda_j (0)}{\lambda_j}\right)^{\frac{\mu_j}{2}}\frac{q^{(3/2+\mu_1 ,\dots ,3/2+\mu_{n})}_{2t} (\lambda^{(n-1)}(0),\lambda^{(n-1)})}{q^{(3/2 ,\dots ,3/2)}_{2t} (\lambda^{(n-1)}(0),\lambda^{(n-1)})}\frac{W^{(3/2 +\mu_1,\dots ,3/2 +\mu_n)}(\lambda^{(n-1)})}{W^{3/2,\dots ,3/2}(\lambda^{(n-1)})} d\mathbb{P}_{\lambda (t)}(\lambda) ,
    \end{align*}
    where $\mathbb{P}_{\lambda (t)}$ is the law of $\lambda (t)$.
    It then follows by dominated convergence and formula \eqref{kernel jacobi simplex} that
    \begin{align*}
        \lim_{t\rightarrow\infty}\mathbb{E}\left( e^{i\sum_{j=1}^n\sum_{a=1}^3u_j^a\frac{\mathfrak{a}_j(t)}{\sqrt{t}}}\right) =e^{-(n-1)\sum_{j=1}^n\sum_{a=1}^3|u_j^a|^2 -\frac{1}{2}\sum_{1\leq j\neq\ell\leq n}\sum_{a=1}^3u_j^au_{\ell}^a } ,
    \end{align*}
    since
    \begin{align*}
        \lim_{t\rightarrow\infty}\sqrt{t(|u_j^1|^2+|u_j^2|^2+|u_j^3|^2) +t^2}-t =\frac{1}{2}(|u_j^1|^2+|u_j^2|^2+|u_j^3|^2).
    \end{align*}
    
    For the second proof we appeal to the classical martingale central limit theorem, see for instance \cite{MR1792301}. Indeed, from Theorem \ref{thm:diffusion-quaternionic-area} we know that the quaternionic stochastic area process $(\mathfrak{a}(t))_{t\geq 0}$ is a martingale with quadratic covariation matrix
    \begin{align*}
        \Sigma (t):=\begin{pmatrix}
            \int_0^t\frac{1-\lambda_1 (s)}{\lambda_1 (s)}ds & \dots & t\\
            \vdots & \ddots & \vdots \\
            t & \dots & \int_0^t\frac{1-\lambda_n (s)}{\lambda_n (s)}ds
        \end{pmatrix}\otimes I_3 .
    \end{align*}
    Now, each of the $(\lambda_j(t))_{t\geq 0}$ is a Jacobi diffusion on the interval $[0,1]$ with generator
    \[
    2\lambda_j(1-\lambda_j )\frac{\partial^2}{\partial\lambda_j^2} +2(2-2n\lambda_j )\frac{\partial}{\partial\lambda_j }
    \]
    and invariant probability measure $\pi(d\lambda)=(2n-1)(2n-2)\lambda (1-\lambda)^{2n-3} d\lambda$. Since $(\lambda_j(t))_{t\geq 0}$ is a recurrent diffusion and $\lambda\to \frac{1-\lambda}{\lambda}$ is in $L^1(\pi)$, one deduces from Birkhoff's ergodic theorem that a.s.
    \[
    \frac{1}{t} \int_0^t\frac{1-\lambda_1 (s)}{\lambda_1 (s)}ds \to \int_0^1 \frac{1-\lambda}{\lambda} \pi(d\lambda)=2n-2\textrm{ as }t\rightarrow\infty .
    \]
    Therefore, a.s.
    \[
\frac{\Sigma(t)}{t} \to \Sigma\textrm{ as }t\rightarrow\infty
    \]
    and from the martingale central limit theorem, in distribution,  one has 
    \begin{align*}
        \frac{\mathfrak{a}(t)}{\sqrt{t}}\rightarrow\mathcal{N}_{3n}(0,\Sigma )\textrm{ as }t\rightarrow\infty.
    \end{align*}
\end{proof}

\begin{remark}
    The sum $\frac{1}{\sqrt{t}}\sum_{j=1}^n\mathfrak{a}_j(t)$ converges in distribution to $\mathcal{N}_3(0,n(n-1)I_3)$ as $t\rightarrow\infty$.
\end{remark}
\begin{remark}
Using the second proof of Theorem \ref{thm:limit-quaternionic-stochastic-area} and the Rebolledo's functional central limit theorem \cite{rebolledo} we actually get a much stronger statement which is the following convergence in distribution of processes when $n \to \infty$:
\[
\frac{1}{\sqrt{n}} (\mathfrak{a}_{nt})_{t \ge 0} \to (W^\Sigma_t)_{t \ge 0} ,
\]
where $(W^\Sigma_t)_{t \ge 0}$ is a Brownian motion on $\mathbb{R}^{3n}$ with covariance matrix $\Sigma t$.
\end{remark}

\section{Simultaneous quaternionic Brownian winding on the quaternionic sphere}\label{sec:quaternionic-winding-sphere}

Consider the $4n-1$ dimensional sphere $\mathbb{S}^{4n-1}\subseteq\mathbb{H}^n$ and a Brownian motion
\begin{align*}
    (X_1(t),\dots ,X_n(t))_{t\geq 0}
\end{align*}
on it.
Assuming that $X_j(0)\neq 0$ for $1\leq j\leq n$, we can consider the polar decompositions
\begin{align*}
    X_j(t)=\varrho_j(t)\Xi_j(t)\textrm{ for } 1\leq j\leq n,
\end{align*}
where $(\rho_j(t))_{t\geq 0}$ are continuous real valued processes with $\varrho (0)>0$ and $(\Xi_j (t))_{t\geq 0}$ continuous $\mathrm{Sp}(1)$-valued processes for $1\leq j\leq n$.
Using our previous results, we want to understand the joint distribution of the quaternionic winding processes $\eta (t):=(\eta_1 (t),\dots ,\eta_n (t))$ defined by
\begin{align*}
    \eta_j (t):=\int_0^t(\circ d\Xi_j (s)) \Xi^{-1}_j(s)
\end{align*}
and study its asymptotics as $t\rightarrow\infty$. Note that the group $\mathrm{Sp}(1)^n$ acts isometrically on $\mathbb{S}^{4n-1}$ by
\begin{align*}
    (g_1,\dots ,g_n)\!\cdot\! (q_1,\dots ,q_n):=(q_1g_1 ,\dots ,q_ng_n).
\end{align*}
This yields a  fibration of homogeneous spaces
\begin{align}\label{eq:fibration-simulteneous-winding-quaternionic-sphere}
    \mathrm{Sp}(1)^n\rightarrow\mathbb{S}^{4n-1}\rightarrow\frac{\mathrm{Sp}(n)}{\mathrm{Sp}(n-1)\times\mathrm{Sp}(1)^n} .
\end{align}
We will denote the corresponding canonical projection by $\pi$. Similarly we have, for $1\leq j\leq n$, the action of $\mathrm{Sp}(1)$ on $\mathbb{S}^{4n-1}$ given by
\begin{align*}
    g\cdot (q_1,\dots ,q_n):=(q_1,\dots ,q_jg,\dots ,q_n) .
\end{align*}
The corresponding projection will be denoted by $\pi_j$.
We will define the $j^{\textrm{th}}$-vertical space of the action by $\mathcal{V}_j :=\ker (d\pi_j )$.
Note that the vertical space of the fibration \eqref{eq:fibration-simulteneous-winding-quaternionic-sphere}, i.e. $\mathcal{V}:=\ker (d\pi )$, is given by $\mathcal{V}=\mathcal{V}_1\oplus\dots\oplus\mathcal{V}_n$.
Due to the fibration \eqref{eq:fibration-simulteneous-winding-quaternionic-sphere} the standard Riemannian metric on $\mathbb{S}^{4n-1}$ can be decomposed orthogonally as
\begin{align*}
    g_{\mathbb{S}^{4n-1}} =g_{\mathcal{H}}\oplus g_{\mathcal{V}_1}\oplus\dots\oplus g_{\mathcal{V}_n} .
\end{align*}
Let $\mu :=(\mu_1 ,\dots ,\mu_n) $ be a multi-index with $\mu_j >0$ for all $1\leq j\leq n$, the canonical variation of the metric is given by
\begin{align*}
    g_{\mu} :=g_{\mathcal{H}}\oplus\frac{1}{\mu_1^2} g_{\mathcal{V}_1}\oplus\dots\oplus\frac{1}{\mu_n^2} g_{\mathcal{V}_n} 
\end{align*}
and the corresponding Laplace-Beltrami operator by
\begin{align*}
    \Delta_{\mathbb{S}_{\mu}^{4n-1}} =\Delta_{\mathcal{H}} +\sum_{j=1}^n\mu_j^2\Delta_{\mathcal{V}_j},
\end{align*}
where $\Delta_{\mathcal{H}}$ denotes the horizontal Laplacian of the fibration \eqref{eq:fibration-simulteneous-winding-quaternionic-sphere} and $\Delta_{\mathcal{V}_j}$ are the vertical Laplacians.
We will denote $\mathbb{S}^{4n-1}$ with the metric $g_{\mu}$ by $\mathbb{S}^{4n-1}_{\mu}$.

\begin{theorem}\label{eq:BM-on-the-quaternionic-sphere}
    Let $(w(t))_{t\geq 0}$ be a Brownian motion on the quaternionic full flag manifold $F_{1,2,\dots ,n-1}(\mathbb{H}^n)$ with quaternionic stochastic area process $(\mathfrak{a}(t))_{t\geq 0}$ and $(\beta (t))_{t\geq 0}$ be a Brownian motion on $\mathrm{Sp}(1)^n$ independent from $(w(t))_{t\geq 0}$.
    The $\mathbb{S}^{4n-1}$-valued process
    \begin{align*}
        X_{\mu}(t):=\begin{pmatrix}
            \frac{\Theta_1 (t)\beta_1(\mu_1^2 t)}{\sqrt{1+|w_1(t)|^2}},\dots ,\frac{\Theta_n (t)\beta_n(\mu_n^2 t)}{\sqrt{1+|w_n(t)|}}
        \end{pmatrix}
    \end{align*}
    is a Brownian motion on $\mathbb{S}^{4n-1}_{\mu}$, where $(\Theta (t))_{t\geq 0}$ is the process defined in Proposition \ref{prop:horizontal-BM-symplectic-group}.
\end{theorem}
\begin{proof}

From Corollary \ref{cor:BM-symplectic-group-in-terms-of-w} and the fact that the last row projection is a Riemannian submersion with totally geodesic fibers we deduce that the process
\begin{align*}
        \begin{pmatrix}
            \frac{\Theta_1 (t) \beta_1(t)}{\sqrt{1+|w_1(t)|^2}},\dots ,\frac{\Theta_n (t)\beta_n(t)}{\sqrt{1+|w_n(t)|}} 
        \end{pmatrix}_{t\geq 0} .
    \end{align*}
is a Brownian motion on $\mathbb{S}^{4n-1}$, i.e. a diffusion with generator
 \begin{align*}
        \Delta_{\mathbb{S}^{4n-1}} =\Delta_{\mathcal{H}} +\sum_{j=1}^n\Delta_{\mathcal{V}_j}
    \end{align*}
   with the notation from before. The result now follows from a simple time-change of $\beta_j$ which rescales $\Delta_{\mathcal{V}_j}$ .
\end{proof}

We deduce the following corollary.

\begin{corollary}
    Let $(X_{\mu}(t))_{t\geq 0}:=(X^1_{\mu}(t),\dots ,X^n_{\mu}(t))_{t\geq 0}$ be a Brownian motion on $\mathbb{S}^{4n-1}_{\mu}$ such that $X^j(0)\neq 0$ for $1\leq j\leq n$.
    Let $(\eta^1_{\mu}(t),\dots ,\eta^n_{\mu}(t))_{t\geq 0}$ be the $\mathfrak{sp}(1)^n$-valued stochastic processes defined by
    \begin{align*}
        \eta^j_{\mu}(t) :=\int_0^t (\circ d\Xi_j (s))\Xi_j(s)^{-1}\textrm{ for }1\leq j\leq n,
    \end{align*}
    where $(\Xi_j (t))_{t\geq 0}$ is the $\mathrm{Sp}(1)$-valued process such that
    \begin{align*}
        X^j_{\mu}(t) =|X^j_{\mu}(t)|\Xi_j (t)\textrm{ for }1\leq j\leq n.
    \end{align*}
    The following convergence holds in distribution
    \begin{align*}
        \frac{1}{\sqrt{t}}(\eta^1_{\mu} (t),\dots ,\eta^{n}_{\mu}(t))\rightarrow\mathcal{N}(0,\Sigma +\mathrm{diag} (\mu_1 I_3,\dots ,\mu_nI_3))\textrm{ as }t\rightarrow\infty ,
    \end{align*}
    where $\mathcal{N}(0,\Sigma +\mathrm{diag} (\mu_1 I_3,\dots ,\mu_nI_3))$ is a multivariate normal distribution of mean $0$ and covariance $\Sigma +\mathrm{diag} (\mu_1 I_3,\dots ,\mu_nI_3)$ with $\Sigma$ as in \eqref{eq:covariance-limit}.
\end{corollary}
\begin{proof}
    From Theorem \ref{eq:BM-on-the-quaternionic-sphere} we have in distribution
    \begin{align*}
        \Xi_j (t) =\Theta_j(t)\beta_j(\mu_j^2 t)
    \end{align*}
    and therefore, for some Brownian motion $(A(t))_{t\geq 0}$ on $\mathfrak{sp}(1)^n$ independent from $(\Theta(t))_{t \ge 0}$,
    \begin{align*}
        \eta^j_{\mu}(t) =&\int_0^t(\Theta_j (s)\circ d\beta_j (\mu_j^2 s) +\circ d\Theta_j(s)\beta_j(\mu_j^2 s)) \beta_j^{-1}(\mu_j^2 s)\Theta^{-1}_j(s)\\
        =&\mu_j \int_0^t\Theta_j(s) \circ dA_j( s)\Theta^{-1}_j(s) -\mathfrak{a}_j(t) \\
        =& \mu_j \int_0^t\Theta_j(s)  dA_j( s)\Theta^{-1}_j(s) -\mathfrak{a}_j(t) ,
    \end{align*}
where in the last step, to go from the Stratonovitch integral to the Itō integral, we used that $(\Theta (t))_{t\geq 0}$ and $(A(t))_{t\geq 0}$ are independent. 
Note that $\mathrm{Ad}$ is an isometry, since the Riemannian metric on $\mathrm{Sp}(1)$ is bi-invariant. It follows that 
$(\eta_{\mu}(t))_{t\geq 0}$ is a martingale with quadratic covariation matrix
\begin{align*}
\begin{pmatrix}
            \mu_1 t+\int_0^t\frac{1-\lambda_1 (s)}{\lambda_1 (s)}ds & \dots & t\\
            \vdots & \ddots & \vdots \\
            t & \dots & \mu_n t+\int_0^t\frac{1-\lambda_n (s)}{\lambda_n (s)}ds
        \end{pmatrix}\otimes I_3 .
   \end{align*} 
    The conclusion now follows from the martingale central limit theorem as in the  proof of Theorem \ref{thm:limit-quaternionic-stochastic-area}.
\end{proof}

\bibliographystyle{plain}
\bibliography{reference}

\vspace{5pt}
\noindent
\begin{minipage}{\textwidth}
    \small
    \textbf{Fabrice Baudoin:} \\
    Department of Mathematics, Aarhus University \\
    Email: fbaudoin@math.au.dk
\end{minipage}

\vspace{10pt}

\noindent
\begin{minipage}{\textwidth}
    \small
    \textbf{Teije Kuijper:} \\
    Department of Mathematics, Aarhus University \\
    Email: t.kuijper@math.au.dk
\end{minipage}

\vspace{10pt}

\noindent
\begin{minipage}{\textwidth}
    \small
    \textbf{Jing Wang:} \\
    Department of Mathematics, Purdue University \\
    Email: jingwang@purdue.edu
\end{minipage}

\end{document}